\documentclass[11pt, letterpaper]{amsart}
\usepackage[margin=1in]{geometry}
\usepackage{microtype}
\usepackage[abbrev]{amsrefs}
\usepackage{xcolor}
\usepackage{mathtools}
\usepackage{amssymb}
\usepackage{mathrsfs}
\usepackage{dsfont}
\usepackage{esint}
\usepackage{amsthm}
\usepackage{longtable}
\usepackage{hyperref}
\hypersetup{
  colorlinks   = true, 
  urlcolor     = blue, 
  linkcolor    = blue, 
  citecolor   = red 
}
\usepackage[normalem]{ulem}

\theoremstyle{plain}
 \newtheorem{theorem}{Theorem}[section]
 \newtheorem*{nonum-theorem}{}

 \newtheorem{lemma}[theorem]{Lemma}
 \newtheorem{corollary}[theorem]{Corollary}
 \newtheorem{proposition}[theorem]{Proposition}

\newcommand{\thistheoremname}{}
\newtheorem*{genericthm*}{\thistheoremname}

\makeatletter
\newenvironment{namedthm*}[2]
  {\renewcommand{\thistheoremname}{#1}%
   \begin{genericthm*}%
   \def\@currentlabelname{#1}
   \textrm{$($}#2\textrm{$)$}}%
  {\end{genericthm*}}
\makeatother

\theoremstyle{definition}
 \newtheorem{definition}[theorem]{Definition}
 
 \newtheorem*{ack}{Acknowledgment}
\theoremstyle{remark}
 \newtheorem{remark}[theorem]{Remark}

\DeclareMathOperator{\mk}{M\MRkern K}
\newcommand{\MRkern}{%
  \mkern-5.8mu
  \mathchoice{}{}{\mkern0.2mu}{\mkern0.5mu}%
}

\newcommand{\scrmk}[1][\sigma]{\mathcal{M\MRkern K}^{#1}}

\def\R{\mathbb{R}}
\def\N{\mathbb{N}}
\renewcommand{\S}{\mathbb{S}}
\def\rad{R}

\DeclareMathOperator{\Id}{Id}
\DeclareMathOperator{\Isom}{Isom}

\DeclareMathOperator{\dist}{\mathrm{d}}
\numberwithin{equation}{section}

\AtBeginDocument{%
   \def\MR#1{}
}

\begin{document}
\title[]{Disintegrated optimal transport \\for metric fiber bundles}
\author[]{Jun Kitagawa}
\address{Department of Mathematics, Michigan State University, East Lansing, MI 48824, USA}
\email{kitagawa@math.msu.edu}
\author[]{Asuka Takatsu}
\address{Graduate School of Mathematical Sciences, The University of Tokyo, Tokyo {153-8914}, Japan \&
 RIKEN Center for Advanced Intelligence Project (AIP), Tokyo {103-0027}, Japan.}
\email{asuka-takatsu@g.ecc.u-tokyo.ac.jp}
%
\begin{abstract}
We define a new two-parameter family of metrics on subsets of Borel probability measures on general metric fiber bundles, 
called the \emph{disintegrated Monge--Kantorovich metrics}. 
This family contains the classical 
Monge-Kantorovich metrics, linearized optimal transport distance, and fibered Wasserstein distances, and certain cases admit isometric embeddings of the sliced and max-sliced Wasserstein spaces. We prove these metrics are complete, separable (except an endpoint case), and geodesic, with a dual representation. 
 Our results cannot be obtained by applying the theory of $L^q$ maps valued in spaces of probability measures, in fact the $L^q$ map case can be recovered from our results by taking the underlying bundle as a trivial product bundle, and the geodesicness and duality results are new even in the fibered Wasserstein case.\\[-20pt]
\end{abstract}
\maketitle
\section{Introduction}
If $(X,\dist_X)$ is a complete, separable metric space, and $1\leq p<\infty$, we let $\mathcal{P}(X)$ denote the Borel probability measures on~$X$, and $\mathcal{P}_p(X)$ the subset of $\mathcal{P}(X)$ with finite $p$th moment. Then the \emph{$p$-Monge--Kantorovich metric} $\mk_p^{X}$ on $\mathcal{P}_p(X)$ is defined via optimal transport theory. To be precise, for $\mu, \nu\in \mathcal{P}_p(X)$,
\begin{align*}
 \Pi(\mu,\nu)\coloneqq &\{\gamma\in \mathcal{P}(X\times X)\mid \gamma(A\times X)=\mu(A),\ \gamma(X\times A)=\nu(A),\text{ for any Borel } A \subset X\},\\
 \mk_p^X(\mu, \nu)\coloneqq &
 \inf_{\gamma\in \Pi(\mu, \nu)}\left\|\dist_X\right\|_{L^{p}(\gamma)}
 =\inf_{\gamma\in \Pi(\mu, \nu)}\left(\int_{X\times X} \dist_X(x,y)^pd\gamma(x, y)\right)^{\frac1p}.
\end{align*}
By~\cite{Villani09}*{Theorem~4.1}, the infimum above is always attained and a minimizer is called a \emph{$p$-optimal coupling} between $\mu$ and $\nu$. It is also well-known that $(\mathcal{P}_p(X), \mk_p^X)$ is a complete, separable metric space (see~\cite{Villani09}*{Theorem~6.18}). These spaces have rich geometric structure, 
forming the groundwork, for example, of the theory of synthetic Ricci curvature, PDEs on singular spaces, and a wide variety of applications (see, for example, \cite{Villani09}*{Parts II and III},  \cite{Santambrogio15}*{Chapters 4, 7, and 8}, and \cite{Galichon16}). 

In this paper, we will introduce metrics on probability measures on \emph{metric fiber bundles}, which capture transport along individual fibers (hence differ from the classic Monge--Kantorovich metrics). The main motivation for introducing these metrics is to develop a fundamental framework to analyze evolutions which are subject to dynamics that act along fibers, such as certain kinetic equations and heterogeneous gradient flows, on spaces more general than $\R^n$ (such as Riemannian manifolds). An advantage of our framework is the underlying bundle does not need to be a vector bundle, thus one can  consider evolutions driven by dynamics on nonlinear fibers, such as on principle bundles.

We start by giving the definition of metric fiber bundles. For a metric space $(X,\dist_X)$, let $\Isom(X)$ denote the isometry group of $X$.
Recall, an action by a subgroup $G$ of $\Isom (X)$ on $X$ is \emph{effective}
if 
$gx=x$ for all $x\in X$ implies that $g$ is the identity element in $G$. 
\begin{definition}\label{def: metric fiber bundle}
    A \emph{metric fiber bundle} is a triple of metric spaces $(E, \dist_E)$, $(\Omega, \dist_\Omega)$, and $(Y, \dist_Y)$, along with a continuous, surjective map $\pi\colon E\to \Omega$, an open cover $\{U_j\}_{j\in \mathcal{J}}$ of $\Omega$, and corresponding maps $\Xi_j\colon U_j\times Y\to \pi^{-1}(U_j)$ (called \emph{local trivializations}) such that the following properties hold. For each $j\in \mathcal{J}$:
    \begin{itemize}\setlength{\leftskip}{-15pt} 
        \item  $\Xi_j$ is a homeomorphism between $U_j\times Y$ endowed with the product metric, and $\pi^{-1}(U_j)$ with the restriction of $\dist_E$.
        \item 
         $\pi(\Xi_j(\omega, y))=\omega$ for all $(\omega, y)\in U_j\times Y$.
        \item Write $\Xi_{j, \omega}(y)\coloneqq \Xi_j(\omega, y)$ for $\omega\in U_j$. Then for $j'\in \mathcal{J}$ with $U_j\cap U_{j'}\neq \emptyset$, there is a subgroup $G$ of $\Isom(Y)$ acting on $Y$ effectively, and a
        map $g_{j}^{j'}:U_j\cap U_{j'}\to G$ (which is well-defined since $G$ is effective) such that 
\[
\Xi_{j', \omega}^{-1} (\Xi_{j, \omega}(y))=g_{j}^{j'}(\omega) y
\quad \text{for }(\omega,y)\in (U_j\cap U_{j'}) \times Y.
\]
\item For $\omega\in U_j$, the map $\Xi_{j, \omega}\colon Y \to \pi^{-1}(\{\omega\})$ is an isometry.
    \end{itemize}
\end{definition}
Throughout the paper, we fix a metric fiber bundle, denoted by $(E, \Omega, \pi, Y)$, where $(E, \dist_E)$ and $(\Omega, \dist_\Omega)$  are metric spaces, with $E$ complete and separable, and~$\Omega$ complete. Note that $(Y, \dist_Y)$ inherits separability and completeness, while $(\Omega, \dist_\Omega)$ inherits separability from $(E, \dist_E)$.
Furthermore, we make the assumption that
\begin{align}\label{eqn: bounded orbits}
    \text{for each $y\in Y$, the orbit $\{gy\mid g\in G\}$ is a bounded subset of $Y$.}
\end{align}
Trivial bundles, the tangent bundle of any $n$-dimensional Riemannian manifold with the Sasaki metric and $G=O(n)$, and any bundle where $Y$ has bounded diameter or $G$ is compact satisfy~\eqref{eqn: bounded orbits}. 

Next we recall a form of disintegration of measures. For $\mu\in \mathcal{P}(X)$ and a Borel map $T$ from $X$ to a measurable space~$Z$, 
the \emph{pushforward measure} $T_{\sharp} \mu \in \mathcal{P}(Z)$ is defined for a Borel set $A\subset Z$ by 
$
T_{\sharp} \mu (A)\coloneqq \mu(T^{-1}(A))
$.

\begin{namedthm*}{Disintegration Theorem} {\cite{DellacherieMeyer78}*{Chapter III-70 and 72}}
\label{thm: disintegration}
Given $\mathfrak{m}\in \mathcal{P}(E)$,
there exists a map $\mathfrak{m}^{\bullet}\colon\Omega \to \mathcal{P}(E)$, uniquely defined $\pi_\sharp\mathfrak{m}$-a.e., such that 
 if $A\subset E$ is Borel, the real valued function on~$\Omega$ defined by
 $\omega \mapsto \mathfrak{m}^\omega(A)$
 is Borel, and
\begin{align*}
 \mathfrak{m}(A)=\int_\Omega\mathfrak{m}^\omega(A)d\pi_\sharp\mathfrak{m}(\omega),\qquad 
    \mathfrak{m}^\omega(E\setminus \pi^{-1}(\{\omega\}))=0\quad\text{for }\pi_\sharp\mathfrak{m}\text{-a.e. }\omega\in \Omega.
\end{align*}
We refer to this as the \emph{disintegration of $\mathfrak{m}$ with respect to $\pi$} and by an abuse of notation, write  $\mathfrak{m}=\mathfrak{m}^{\bullet}\otimes(\pi_\sharp\mathfrak{m})$. 
\end{namedthm*}
We now fix a Borel probability measure $\sigma$ on $\Omega$, and for $1\leq p<\infty$ define
\begin{align*}
\mathcal{P}^\sigma(E)
&\coloneqq \left\{
\mathfrak{m}\in \mathcal{P}(E)\bigm| \pi_\sharp \mathfrak{m}=\sigma
\right\},\\
    \mathcal{P}^\sigma_p(E)&\coloneqq \{\mathfrak{m}=\mathfrak{m}^\bullet\otimes \sigma\in \mathcal{P}^\sigma(E)\mid \mathfrak{m}^\omega\in \mathcal{P}_p(\pi^{-1}(\{\omega\}))\text{ for $\sigma$-a.e. $\omega$}\}.
    \end{align*}
\begin{definition}
Let $1\leq p <\infty$ and $1\leq q\leq\infty$.
Given $\mathfrak{m}$, $\mathfrak{n}\in\mathcal{P}^\sigma_p(E)$,
we define
\begin{align*}
 \scrmk_{p,q}(\mathfrak{m}, \mathfrak{n})\coloneqq \left\| \mk_p^{E}(\mathfrak{m}^\bullet, \mathfrak{n}^\bullet)\right\|_{L^q(\sigma)},
\end{align*}
and call $\scrmk_{p, q}$ the \emph{disintegrated $(p, q)$-Monge--Kantorovich metric}.
\end{definition}
By \cite{AmbrosioGigliSavare08}*{Lemma 12.4.7}, for $\mathfrak{m}$, $\mathfrak{n}\in \mathcal{P}^\sigma_p(E)$ the function $\omega\mapsto \mk_p^E(\mathfrak{m}^\omega, \mathfrak{n}^\omega)$ is Borel, hence $\scrmk_{p, q}$ is well-defined. So that $\scrmk_{p, q}$ is a true metric, we will restrict to a subset $\mathcal{P}^\sigma_{p,q}(E)\subset \mathcal{P}^\sigma_p(E)$, whose definition is deferred to~\eqref{scr}. We will also have use for an auxiliary function $\dist^p_{E, y_0}\colon \Omega\times E\to [0,\infty)$ for a fixed $y_0\in Y$, whose exact definition is given in \eqref{eqn: dE}. In the special case when $E=\Omega\times Y$ is a trivial bundle, $\mathcal{P}^\sigma_{p,q}(E)$ and $\dist^p_{E, y_0}$ have the simple expressions,
\begin{align*}
    \mathcal{P}^\sigma_{p,q}(E)
    &\coloneqq\{\mathfrak{m}\in \mathcal{P}^\sigma_p(E)\mid \lVert \mk_p^Y( \delta^Y_{y_0},\mathfrak{m}^\bullet)\rVert_{L^q(\sigma)}<\infty\},\\
    \dist^p_{E, y_0}(\omega, u)
&\coloneqq
\dist_E( (\omega,y_0), u )^p
\quad\text{for }(\omega,u)\in \Omega \times E.
\end{align*}

To state our main results, we give the definitions of a few function spaces. 

\begin{definition}\label{C0}
For a locally compact Hausdorff space $X$, a real valued function $\phi$ on $X$ is said to \emph{vanish at infinity} if
$\left\{ x\in X \bigm| \left|\phi(x)\right| \geq \varepsilon \right\}$ 
is compact for any $\varepsilon>0$. 
We let $C_0(X)$ and $C_b(X)$ stand for 
the space of continuous functions on $X$
vanishing at infinity and the space of bounded continuous functions on $X$ respectively, both equipped with the supremum norm.
\end{definition}
We also define
\begin{align*}
 \mathcal{X}_p&\coloneqq\left\{\xi\in C(E)\Bigm| \frac{\xi}{1+\dist^p_{E,y_0}(\pi, \cdot)}\in C_0(E)\right\}\quad
\text{with }\left\| \xi\right\|_{\mathcal{X}_p}\coloneqq\sup_{u\in E}\frac{\left| \xi(u)\right|}{1+\dist^p_{E,y_0}(\pi(u), u)},\\
\mathcal{A}_{p,E,\sigma}
&\coloneqq\left\{
(\Phi,\Psi)\in C_b(E)\times C_b(E)
\Biggm|
-\Phi(u)-\Psi(v) \leq \dist_E(u, v)^p
\text{ for }u, v\in E \text{ with } \pi(u)=\pi(v)
\right\},\\
\mathcal{Z}_{r', \sigma}
&\coloneqq 
\left\{ \zeta\in C_b(\Omega) \bigm| \left\| \zeta\right\|_{L^{r'}(\sigma)}\leq 1,\ \zeta>0
\right\} \text{ for } r'\in [1, \infty].
\end{align*}
Additionally, for $\xi\in \mathcal{X}_p$, define
\begin{align*}
S_p\xi(u)
\coloneqq \sup_{v\in \pi^{-1}(\{\pi(u)\})} \left(-\dist_E(u, v)^p-\xi(v) \right)
\quad
\text{for }u\in E;
\end{align*}
 as a supremum of continuous functions, we see $S_p\xi$ is Borel on $E$ for any $\xi\in \mathcal{X}_p$.

We also recall the following definitions.
\begin{definition}\label{def: geodineq}
Let $(X,\dist_X)$ be a metric space.
A curve $\rho\colon[0,1]\to X$ is a
 \emph{minimal geodesic} if 
\begin{equation*}
\dist_X(\rho(\tau_1), \rho(\tau_2))= |\tau_1-\tau_2|\dist_X(\rho(0),\rho(1))\text{ for any }\tau_1,\ \tau_2\in [0,1].
\end{equation*}
A metric space $(X,\dist_X)$ is \emph{geodesic} 
if any two points in $X$ can be joined by a minimal geodesic.
A geodesic space $(X, \dist_X)$ is \emph{ball convex with respect to a point $x_0\in X$} if for any minimal geodesic $\rho\colon [0, 1]\to X$ and $\tau\in [0, 1]$
    \begin{align*}
        \dist_X(\rho(\tau), x_0)\leq \max\{\dist_X(\rho(0), x_0), \dist_X(\rho(1), x_0)\}.
    \end{align*}
\end{definition}

With the above in hand, our main result is the following.
\begin{theorem}\label{thm: main disint}
    Let $1\leq p<\infty$, $1\leq q\leq \infty$. 
    Let $(E, \Omega, \pi, Y)$ be a metric fiber bundle satisfying~\eqref{eqn: bounded orbits}, with $(E, \dist_E)$ complete and
    separable, and $(\Omega, \dist_\Omega)$ complete, and let $\sigma\in\mathcal{P}(\Omega)$. 
    Then:
    \begin{enumerate}
\setlength{\leftskip}{-15pt}
        \item\label{thm: disint complete} 
   $(\mathcal{P}^\sigma_{p,q}(E), \scrmk_{p, q})$ is a complete metric space. It is also separable when $q<\infty$. 
      \item\label{thm: disint geodesic} If  
      $(Y, \dist_Y)$ is a geodesic space that is ball convex with respect to some point in $Y$,    
then $(\mathcal{P}^\sigma_{p,q}(E), \scrmk_{p, q})$ \emph{is} geodesic.
        \item\label{thm: disint duality} Let  $p\leq q$, set $r\coloneqq q/p$, and 
denote by $r'$ the H\"older conjugate of $r$.
Then if $(Y, \dist_Y)$ is locally compact, for $\mathfrak{m}$, $\mathfrak{n}\in \mathcal{P}^\sigma_p(E)$ we have 
\begin{align*}
\scrmk_{p, q}(\mathfrak{m}, \mathfrak{n})^p\
&=\sup
\left\{
-\int_E
(\zeta\circ \pi)\Phi d\mathfrak{m}
-\int_E(\zeta\circ \pi)\Psi d\mathfrak{n}
\biggm|
(\Phi, \Psi)\in\mathcal{A}_{p,E,\sigma},\ \zeta\in \mathcal{Z}_{r', \sigma}
\right\}.
\end{align*}
If $(E, \dist_E)$ is locally compact, we also have
\begin{align*}
\scrmk_{p, q}(\mathfrak{m}, \mathfrak{n})^p
&=\sup
\left\{
-\int_E (\zeta\circ \pi)
(S_p\Psi) d\mathfrak{m}
-\int_E(\zeta\circ \pi)\Psi d\mathfrak{n}
\biggm|
\Psi\in \mathcal{X}_p\cap C_b(E),\
\zeta\in \mathcal{Z}_{r', \sigma}
\right\}.
\end{align*}
    \end{enumerate}
\end{theorem}
\begin{remark}\label{rem: p=q duality}
    It can be seen that when $p=q$ (i.e., $r'=\infty$), the maximum value in Theorem~\ref{thm: main disint}~\eqref{thm: disint duality} is attained by $\zeta\equiv 1$, hence the supremum over $\zeta$ is not actually needed in this case. 
\end{remark}
Note that, unlike the case when $E$ is a trivial bundle, our main results \emph{do not} follow from the analysis of $L^q$ maps into spaces of probability measures. Additionally, the geodesicness and duality results, Theorem~\ref{thm: main disint}~\eqref{thm: disint geodesic} and~\eqref{thm: disint duality} are new even in the case of trivial bundle $E$.

\subsection*{Motivation and  
literature}\label{subsec: motivation}
Our disintegrated Monge--Kantorovich metrics are the first such construction on truly general fiber bundles. If $E=\{\omega_0\}\times Y$ for some one point set $\{\omega_0\}$ and $\sigma=\delta_{\omega_0}$, then $(\mathcal{P}^\sigma_{p, q}(E), \scrmk_{p, q})$ is exactly the space $(\mathcal{P}_p(Y), \mk_p^Y)$.  Another simple subcase is when $E=\Omega\times Y$ is a trivial bundle (i.e., $G$ is the trivial group, and there is only one local trivialization map with a cover of $\Omega$ by only one set). In~\cite{PeszekPoyato23}, the authors introduce the \emph{fibered quadratic Wasserstein distance}, which corresponds to our $\scrmk_{2, 2}$ on the trivial bundle $E=\R^n\times \R^n$. When $E=\Omega\times Y$ is a trivial bundle, it is possible to view $(\mathcal{P}^\sigma_{p, q}(E), \scrmk_{p, q})$ as the metric space valued $L^q$ space on~$(\Omega, \sigma)$ 
where the range is $(\mathcal{P}_p(Y), \mk_p^Y)$ (i.e., elements are of the form $\omega\mapsto \mathfrak{m}^\omega$).
Properties such as completeness for such spaces are claimed in various works, but do not appear to come with proofs in the literature except when the range is a Banach space (i.e., for Bochner--Lebesgue spaces), or when $p=2$ (see~\cite{PeszekPoyato23}*{Appendix~A}). No such identification with a metric space valued $L^q$ space is available when $E$ is a general metric fiber bundle, hence the jump from product structure to general fiber bundle is highly nontrivial, and in particular the methods of \cite{PeszekPoyato23} {cannot} be extended to our general case. However as demonstrated in that paper, already in $\R^n\times \R^n$, there are a multitude of applications to analysis of gradient flows with heterogeneous structure, such as the Kuramoto--Sakaguchi equation and the multi-species Patlak--Keller--Segel model. Our metric will open up the possibility of considering such evolutions on manifolds, or more singular metric spaces. 

We also note that our metrics $\scrmk_{p, q}$ are related to a notion of \emph{measure differential equation} introduced in~\cite{Piccoli19}. There, a notion of flows generated by  probability measure fields (as opposed to vector fields) is introduced and analyzed in a systematic way; among other applications, they are raised as natural candidates for mean-field limits in the setting of multi-particle systems. A quantity $\mathcal{W}(V_1, V_2)$ between probability measures $V_1$ and $V_2$ on the tangent bundle of $\R^n$ is defined in~\cite{Piccoli19}*{Definition 4.1}. It is noted that $\mathcal{W}$ is in general \emph{not} a metric, but in the special case when $V_1$ and $V_2$ have the same marginal when projected onto the base space, $\mathcal{W}$ exactly equals our $\scrmk_{1, 1}$, hence does give a metric. In particular, $\scrmk_{p, q}$ can be used as a pointwise metric between probability measure fields as defined in~\cite{Piccoli19}*{Definition 2.1}, hence could be of use in the analysis of the stability of families of measure differential equations.

We also mention that our family of metrics have potential applications toward the development of a variational framework for spatially inhomogeneous kinetic equations. We are currently working on such a framework starting with the case of the Kolmogorov equation 
\begin{align*}
    \partial_t f(t, x, v)+\langle v, \nabla_x f(t, x, v)\rangle=\Delta_v f(t, x, v),
     \quad  (t,x,v)\in (0,\infty) \times \mathbb{R}^n \times \mathbb{R}^n
\end{align*}
(see, for example, \cite{ImbertSilvestre20}*{Section~2.1}).
One can view the Kolomogorov equation as a PDE on the tangent bundle of $\R^n$ whose key feature is transport on the base space coupled with diffusion in each fiber, such structure is amenable to the geometry induced by our disintegrated metrics, but care must be taken to describe the evolution in the base space. Such a framework also potentially leads to analysis of the analogue of the Kolmogorov equation on manifolds other than Euclidean space; this will be detailed in a forthcoming work.

When $E=\Omega\times \Omega$ where $\Omega\subset \R^n$ is a suitable set, $\sigma\in\mathcal{P}_p(\Omega)$ is absolutely continuous with respect to $n$-dimensional Lebesgue measure, and $\mathfrak{m}$, $\mathfrak{n}$ are $p$-optimal couplings between $\sigma$ and measures $\mu$,  $\nu\in\mathcal{P}_p(\Omega)$ respectively, 
it can be seen that $\scrmk_{p, p}(\mathfrak{m}, \mathfrak{n})$ coincides with (an extension from the case $p=2$ of) the \emph{linearized optimal transport metric} introduced in~\cite{WangSlepcevetal13}*{Section~2.3} between the right marginals of $\mathfrak{m}$ and $\mathfrak{n}$. 
This can be used to obtain properties of the linearized optimal transport metric, for example, Proposition~\ref{prop: optimal closed} below  yields that the linearized optimal transport metric is  complete. This claim is nontrivial, as it shows that the subset of optimal mappings from $\sigma$ is closed in $L^p(\sigma)$. 
We also note there is a somewhat similar notion of \emph{layerwise-Wasserstein distance} introduced in~\cite{KimPassSchneider20}*{Definition~2.2
}.

The disintegrated metrics are also related to the \emph{sliced Wasserstein} and \emph{max-sliced Wasserstein} metrics (see~\cite{sliced-original}*{Section~2.2} and \cite{max-sliced19}*{Definition~2}), which can be isometrically embedded into certain disintegrated Monge--Kantorovich spaces, as we show below in Proposition~\ref{thm: disint embed}. This is of interest if one is interested in applying the gradient flow theory on sliced Monge--Kantorovich spaces, as these are generally \emph{not} geodesic spaces (as shown in~\cite{KitagawaTakatsu24a}*{Main Theorem}) but the disintegrated Monge--Kantorovich spaces are geodesic. Another approach to gradient flows on sliced Monge--Kantorovich spaces has been proposed in~\cite{ParkSlepcev23}*{Section~7}, where the length space structure generated by the sliced metric for $p=q=2$ is considered instead.

The rest of this paper is organized as follows.
We give some preliminary definitions and notation in Section~\ref{sec: preliminaries}, then present the proof of Theorem~\ref{thm: main disint} in Section~\ref{sec: proofs}. We also prove some supplementary results on the disintegrated Monge--Kantorovich metrics that do not directly fall under Theorem~\ref{thm: main disint} in Section~\ref{sec: disint etc}.
\section{Preliminary results}\label{sec: preliminaries}
For the remainder of the paper $(E, \Omega, \pi, Y)$ is a metric fiber bundle with some locally finite open cover $\{U_j\}_{j\in \mathcal{J}}$ of $\Omega$, and associated local trivializations $\{\Xi_j\}_{j\in \mathcal{J}}$. We will assume that $(E, \dist_E)$ is a complete, separable metric space, $(\Omega, \dist_\Omega)$ a complete metric space, and $G$ satisfies assumption~\eqref{eqn: bounded orbits}, with other conditions added as necessary. We also fix $\sigma\in \mathcal{P}(\Omega)$. Throughout this paper, 
we will take $1\leq p<\infty$ and $1\leq q \leq \infty$ unless stated otherwise.

Since $(\Omega, \dist_\Omega)$ is a Lindl\"of space by its separability, and is paracompact since it is metric, there is a countable, locally finite subcover~$\{U_j\}_{j\in \N}$ of $\{U_j\}_{j\in \mathcal{J}}$, with associated local trivializations $\{\Xi_j\}_{j\in \N}$. Additionally, we can find a (continuous) partition of unity $\{\chi_j\}_{j\in \N}$ subordinate to~$\{U_j\}_{j\in \N}$. We will write 
\[
U'_j\coloneqq \{\omega\in \Omega\mid \chi_j(\omega)>0\},
\]
which is a nonempty, open set for each $j\in \N$. Since $\{\chi_j\}_{j\in \N}$ is a partition of unity, we see $\{U'_j\}_{j\in \N}$ is an open cover of $\Omega$, then for later use, we define the cover~$\{V_j\}_{j\in \N}$ of $\Omega$ consisting of mutually disjoint Borel sets by 
\begin{align*}
 V_1\coloneqq U'_1,\quad V_j\coloneqq U'_j\setminus\bigcup_{j'=1}^{j-1}V_{j'},\ j\geq 2;
\end{align*} 
by construction $\chi_j>0$ on $V_j$ and $V_j\subset U_j$ for each $j\in \N$. 

For a metric space $(X, \dist_X)$, we will write $B^X_r(x)$ for the open ball centered at $x\in X$ of radius $r>0$. 
If $\mu$ is any Borel measure on a topological space $X$, we will denote by $\mathcal{B}_\mu$ the completion of the Borel $\sigma$-algebra over $X$ with respect to $\mu$. 
We also denote by ~$\mathds{1}_A$ the characteristic function of a set~$A$, and write $\delta^Y_y$ to denote the delta measure at the point $y$ on a space $Y$.

Next we need some measure theoretical preliminaries.
\begin{definition}
If $X$ is any space, we say a map $f\colon \Omega\to X$ is \emph{simple} if there are finite collections $\{\Omega_i\}_{i=1}^I\subset \mathcal{B}_\sigma$ and $\{x_i\}_{i=1}^I\subset X$, such that the $\Omega_i$ form a partition of $\Omega$ and 
\begin{align*}
    f(\omega)=x_i\quad\text{whenever }\omega\in \Omega_i.
\end{align*}
We will denote such a function by 
\begin{align*}
    f=\sum_{i=1}^I\mathds{1}_{\Omega_i}x_i.
\end{align*}

If $(X, \dist_X)$ is a metric space, a map $f\colon \Omega\to X$ is \emph{$\sigma$-strongly measurable} if there exists a sequence of simple functions that converges $\sigma$-a.e. pointwise to $f$. 
We will write $L^0(\sigma; X)$ for the collection of maps from $\Omega$ to $X$ which are $\sigma$-strongly measurable. 

 Also if $Z$ is any measurable space with a $\sigma$-algebra $\mathcal{F}_Z$, we will say a map $f\colon Z\to X$ is \emph{$\mathcal{F}_Z$-measurable} if $f^{-1}(O)\in \mathcal{F}_Z$ for any open set $O\subset X$. If~$Z$ is equipped with a topology and $\mathcal{F}_Z$ is the Borel $\sigma$-algebra on $Z$, then we simply say $f$ is \emph{Borel}.
\end{definition}
Note the above definitions do not actually require any vector space structure on the range $X$, since the sets $\Omega_i$ in the definition of simple are assumed mutually disjoint.
\begin{remark}\label{V58}
    By~\cite{Varadarajan58}*{Theorem 1}, if $(X, \dist_X)$ is separable, a $\mathcal{B}_\sigma$-measurable map $f\colon \Omega\to X$ is $\sigma$-strongly measurable. In the converse direction, since the inverse image of any set under a simple function is a finite union of elements of $\mathcal{B}_\sigma$, a $\sigma$-strongly measurable map is always $\mathcal{B}_\sigma$-measurable (regardless of separability of the range).

    By~\cite{AmbrosioFuscoPallara00}*{Proposition 2.26} (although this proposition is stated for measures on $\R^n$, it is easy to see the proof holds in general metric spaces), if a map $\mu_\bullet: \Omega\to \mathcal{P}_p(X)$ for some metric space $(X, \dist_X)$ satisfies that $\omega\mapsto \mu_\omega(A)$ is a Borel function for any \emph{open}~$A\subset X$, this property is satisfied for any \emph{Borel}~$A\subset X$. Since each $\mu_\omega$ is a probability measure, it is clearly also equivalent to have the above condition hold for any \emph{closed} $A\subset X$ as well. Then by the proof of~\cite{AmbrosioGigliSavare08}*{Theorem 12.4.7}, $\omega\mapsto \mu_\omega$ is Borel as a map from $(\Omega, \dist_\Omega)$ to $(\mathcal{P}_p(X), \mk_p^X)$. 
Since $(\mathcal{P}_p(X), \mk_p^X)$ is separable, the map is also $\sigma$-strongly measurable. On the other hand, it is easy to see that a $\sigma$-strongly measurable map into $(\mathcal{P}_p(X), \mk_p^X)$ satisfies that $\omega\mapsto \mu_\omega(A)$ is Borel for all open (and closed) sets $A$, thus the above are equivalent characterizations of measurability. 
Additionally, if $\mu_\bullet$ is a map satisfying any of the equivalent characterizations of measurability in the previous paragraph, we can define the function
\begin{align*}
    \mu(A)\coloneqq \int_\Omega \mu_\omega(A)d\sigma(\omega)
\end{align*}
for any Borel $A\subset X$. Then for any disjoint collection $\{A_\ell\}_{\ell\in \N}$ of Borel sets in $X$, we have
\begin{align*}
    \mu\left(\bigcup_{\ell\in \N}A_\ell\right)
    &=\int_\Omega \mu_\omega\left(\bigcup_{\ell\in \N}A_\ell\right)d\sigma(\omega)
    =\int_\Omega \sum_{\ell\in \N}\mu_\omega(A_\ell)d\sigma(\omega)=\sum_{\ell\in \N}\mu(A_\ell)
\end{align*}
by monotone convergence. Clearly $\mu(\emptyset)=0$ and $\mu(X)=1$, with $\mu(A)\geq 0$ for any Borel set $A\subset X$, hence we see $\mu\in\mathcal{P}(X)$. These facts will be used freely throughout the remainder of the paper.
\end{remark}
\begin{remark}\label{rem: MK convex}
Let  $K\in \mathbb{N}$ with $K\geq 2$.
    Note that if $\gamma_k\in \Pi(\mu_k, \nu_k)$ for $1\leq k\leq K$, then 
    \[\sum_{k=1}^K\lambda_k\gamma_k\in \Pi\left(\sum_{k=1}^K\lambda_k\mu_k, \sum_{k=1}^K\lambda_{k}\nu_{k}\right)
    \quad 
    \text{for }\sum_{k=1}^K\lambda_k=1 \text{  with }\lambda_k\geq 0.
    \]
    Thus for any metric space $(X, \dist_X)$ and $1\leq p<\infty$, we have
    \begin{align*}
        \mk_p^X\left(\sum_{k=1}^K\lambda_k\mu_k, \sum_{k=1}^K\lambda_k\nu_k\right)^p
        \leq \sum_{k=1}^K\lambda_k\mk_p^X(\mu_k, \nu_k)^p.
    \end{align*}
    Also since each map $\Xi_{j, \omega}$ is an isometry between $Y$ and $\pi^{-1}(\{\omega\})$, for any $1\leq p<\infty$ and measures $\mu$, $\nu\in \mathcal{P}_p(Y)$, we have
    \begin{align*}
        \mk_p^Y(\mu, \nu)=\mk_p^E((\Xi_{j, \omega})_\sharp \mu, (\Xi_{j, \omega})_\sharp \nu)
        \quad \text{for }\omega \in \Omega.
    \end{align*}
    We will freely use these properties in the sequel.
\end{remark}
Now recall that for $\mathfrak{m}\in \mathcal{P}^{\sigma}(E)$,
we write
$
\mathfrak{m}=\mathfrak{m}^{\bullet} \otimes \sigma
$
where $\mathfrak{m}^\omega\in \mathcal{P}(\pi^{-1}(\{\omega\}))$ for each $\omega\in \Omega$, following from \nameref{thm: disintegration}. We will fix some $y_0\in Y$ and for ease of notation, write 
\begin{align*}
\dist_{y_0}(t)\coloneqq \dist_Y(y_0, t)\quad\text{for }t\in Y.
\end{align*}
Then for any Borel $A\subset E$, define
\begin{align}\label{eqn: delta construction}
     (\delta^\bullet_{E, y_0}\otimes \sigma)(A)\coloneqq \sum_{j\in \N}\int_{\Omega}\chi_j(\omega) (\Xi_{j, \omega})_\sharp\delta^Y_{y_0}(A)d\sigma(\omega).
 \end{align}
 If we define $\delta^\omega_{E, y_0}\in \mathcal{P}(E)$ by 
\begin{align*}
\delta^\omega_{E,y_0}\coloneqq \sum_{j\in \N}\chi_j(\omega) (\Xi_{j, \omega})_\sharp\delta^Y_{y_0},
\end{align*}
the following Lemma~\ref{lem: pushforward measurable} implies~\eqref{eqn: delta construction} is an element of $\mathcal{P}^\sigma_p(E)$ whose disintegration with respect to~$\pi$ is actually given by $\delta^\bullet_{E, y_0}\otimes \sigma$.
\begin{lemma}\label{lem: pushforward measurable}
If $\mu\in \mathcal{P}_p(Y)$ for some $1\leq p<\infty$, the functions on the Borel sets of $E$ defined by
\begin{align}
\label{eqn: disjoint pushforward disintegration def}
            A&\mapsto \sum_{j\in \N}\int_{\Omega}\mathds{1}_{V_j}(\omega)(\Xi_{j, \omega})_\sharp\mu(A)d\sigma(\omega),\\
        \label{eqn: pushforward disintegration def}
            A&\mapsto \sum_{j\in \N}\int_{\Omega}\chi_j(\omega)(\Xi_{j, \omega})_\sharp\mu(A)d\sigma(\omega),
        \end{align}
        are elements of $\mathcal{P}^\sigma_{p, q}(E)$ for any $1\leq q\leq \infty$, with disintegrations 
\[
\sum_{j\in \N}\mathds{1}_{V_j}(\Xi_{j, \bullet})_\sharp\mu\otimes \sigma \quad \text{and}\quad
\sum_{j\in \N}\chi_j(\Xi_{j, \bullet})_\sharp\mu\otimes \sigma
\] 
respectively, with respect to $\pi$.
\end{lemma}
\begin{proof}
Fix any $\mu\in \mathcal{P}_p(Y)$ and open set $A\subset E$. Then by Fatou's lemma the function 
\begin{align*}
\omega\mapsto\int_Y\mathds{1}_A(\Xi_{j,\omega}(t))d\mu(t)
\end{align*} 
is lower semi-continuous, in particular Borel, on $U_j$ for any $j\in \N$. Thus we immediately see 
\[
\omega\mapsto 
\sum_{j\in \N}\mathds{1}_{V_j}(\omega)(\Xi_{j, \omega})_\sharp \mu(A)
= \sum_{j\in \N} \mathds{1}_{V_j}(\omega) \int_{Y} \mathds{1}_A(\Xi_{j,\omega}(t)) d\mu(t)
\]
 is Borel for any open set $A\subset E$, hence for any Borel set. Thus~\eqref{eqn: disjoint pushforward disintegration def} is well-defined for any Borel $A\subset E$,
and by Remark~\ref{V58},
\[
\mathfrak{m}=\sum_{j\in \N}\mathds{1}_{V_j}(\Xi_{j, \bullet})_\sharp\mu\otimes \sigma
\]
is a nonnegative probability measure, which we easily see belongs to $\mathcal{P}^\sigma(E)$. 
Also, if $\omega\in \Omega$ and $u_0\in \pi^{-1}(\{\omega\})$ are fixed,
\begin{align*}
\int_E \dist_E (u_0,u)^p d\mathfrak{m}(u)
&=\sum_{j\in \N}\mathds{1}_{V_j}(\omega)\int_E \dist_E(u_0, u)^p d (\Xi_{j, \omega})_\sharp\mu(u)\\
   &=\sum_{j\in \N}\mathds{1}_{V_j}(\omega)\int_Y \dist_E(u_0, \Xi_{j, \omega}(t))^p d \mu(t)
   =\sum_{j\in \N}\mathds{1}_{V_j}(\omega)\int_Y \dist_Y(\Xi^{-1}_{j,\omega}(u_0), t)^p d \mu(t)\\
   &<\infty,
\end{align*}
where the finiteness follows since $\mu\in \mathcal{P}_p(Y)$, and the sum above is finite from disjointness of the sets $V_j$, thus $\mathfrak{m}\in \mathcal{P}^\sigma_p(E)$. 
The same proof holds replacing each $\mathds{1}_{V_j}$ with $\chi_j$, the local finiteness taking the place of disjointness of the sets $V_j$, hence the expression in~\eqref{eqn: pushforward disintegration def} also defines an element of $\mathcal{P}^\sigma_p(E)$; in particular, taking $\mu=\delta^Y_{y_0}$ we also see $\delta^\bullet_{E, y_0}\otimes\sigma$ defined by~\eqref{eqn: delta construction} belongs to ~$\mathcal{P}^\sigma_p(E)$.

Next, fix $\omega\in \Omega$, then using the local finiteness property of the partition of unity $\{\chi_j\}_{j\in \N}$ and recalling Remark~\ref{rem: MK convex}, we have
\begin{align*}
\begin{split}
    \mk_p^E(\delta^{\omega}_{E,y_0}, \mathfrak{m}^\omega)^p
    &=\mk_p^E\left(\sum_{j\in \N}\chi_{j}(\omega)(\Xi_{j, \omega})_\sharp\delta^Y_{y_0}, \sum_{j'\in \N}\mathds{1}_{V_{j'}}(\omega)(\Xi_{j', \omega})_\sharp\mu\right)^p\\
    &\leq \sum_{j, j'\in \N}\chi_{j}(\omega)\mathds{1}_{V_{j'}}(\omega)\mk_p^E((\Xi_{j, \omega})_\sharp\delta^Y_{y_0}, (\Xi_{j', \omega})_\sharp\mu)^p\\
    &=\sum_{j, j'\in \N}\chi_{j}(\omega)\mathds{1}_{V_{j'}}(\omega)\mk_p^Y (g^{j'}_j(\omega)_ \sharp\delta^Y_{y_0}, \mu)^p\\
    &\leq2^{p-1}\sum_{j, j'\in \N}\chi_{j}(\omega)\mathds{1}_{V_{j'}}(\omega)\left(\mk_p^Y(\delta^Y_{y_0}, \mu)^p
    +\mk_p^Y(\delta^Y_{y_0}, g^{j'}_j(\omega)_\sharp\delta^Y_{y_0})^p\right)\\
    &=2^{p-1}\sum_{j, j'\in \N}\chi_{j}(\omega)\mathds{1}_{V_{j'}}(\omega)\left(\mk_p^Y(\delta^Y_{y_0}, \mu)^p
    +\dist_Y(y_0, g^{j'}_{j}(\omega) y_0)^p\right),
\end{split}
\end{align*}
which is bounded independent of $\omega\in \Omega$ since $\mu\in \mathcal{P}_p(Y)$ and by~\eqref{eqn: bounded orbits}.  
Thus $\mathfrak{m}\in \mathcal{P}^\sigma_{p, q}(E)$; an analogous proof applies for~\eqref{eqn: pushforward disintegration def}  and the lemma is proved.
\end{proof}

We also define the function on $\Omega\times E$ by
\begin{align}\label{eqn: dE}
\dist_{E,y_0}^p(\omega, u)
\coloneqq
\sum_{j\in \N}\chi_j(\omega) \dist_E(\Xi_{j, \omega}(y_0),u)^p
\quad \text{for }(\omega, u)\in \Omega\times E,
\end{align}
and the space
\begin{align}\label{scr}
\mathcal{P}^\sigma_{p,q}(E)\coloneqq \left\{ \mathfrak{m}\in \mathcal{P}^\sigma_p(E) \Biggm|
\scrmk_{p, q}(\delta^\bullet_{E, y_0}\otimes \sigma,\mathfrak{m})<\infty
\right\},
\end{align}
this will be the space on which we consider $\scrmk_{p, q}$. 
To finish this section, we show that the definition of $\mathcal{P}^\sigma_{p, q}(E)$ does not depend on the choices of cover, local trivializations, partition of unity, nor choice of fixed point in $Y$.
\begin{lemma}\label{lem: P_pq well defined}
    Let $(E, \Omega, \pi, Y)$ be a metric fiber bundle with  open cover $\{U_j\}_{j\in \mathcal{J}}$ of $\Omega$ and associated local trivializations $\{\Xi_j\}_{j\in \mathcal{J}}$. Then, the definition of $\mathcal{P}^\sigma_{p, q}(E)$ is independent of the choices of subcover $\{U_j\}_{j\in \N}$, $\{\Xi_j\}_{j\in \N}$, partition of unity $\{\chi_j\}_{j\in \N}$, and $y_0$.
\end{lemma}
\begin{proof}
To see this, suppose $\{\widetilde{U}_j\}_{j\in \N}$, $\{\widetilde{\Xi}_j\}_{j\in \N}$, $\{\tilde{\chi}_j\}_{j\in \N}$ are another choice of open subcover, associated local trivializations, and partition of unity, take some other point $\tilde{y}_0\in Y$, and let $\delta^\bullet_{E, \tilde{y}_0}\otimes \sigma$ denote the construction~\eqref{eqn: delta construction} made with these choices. 
Then, for each $\omega \in U_j\cap U_{j'}$ with $j,j'\in\mathbb{N}$,
there exists $\gamma_j^{j'}(\omega)\in G$ such that 
$\widetilde{\Xi}_{j',\omega}^{-1}(\Xi_{j,\omega}(y))=\gamma_j^{j'}(\omega) y$
for $y\in Y$.
By the triangle inequality from Theorem~\ref{thm: main disint}~\eqref{thm: disint complete} below, (which does not rely on this lemma) it is sufficient to show $\scrmk_{p, q}(\delta^\bullet_{E, y_0}\otimes \sigma, \delta^\bullet_{E, \tilde{y}_0}\otimes \sigma)<\infty$. To this end, fix $\omega\in \Omega$,
then
\begin{align*}
    \mk_p^E(\delta^\omega_{E, y_0}, \delta^\omega_{E, \tilde{y}_0})^p
    &=\mk_p^E\left(
    \sum_{j\in \N}\chi_j(\omega) (\Xi_{j, \omega})_\sharp\delta^Y_{y_0}, \sum_{j'\in \N}\tilde{\chi}_{j'}(\omega) (\widetilde{\Xi}_{j', \omega})_\sharp\delta^Y_{\tilde{y}_0}\right)^p\\
    &\leq \sum_{j, j'\in \N}\chi_j(\omega)\tilde{\chi}_{j'}(\omega)\mk_p^E\left( (\Xi_{j, \omega})_\sharp\delta^Y_{y_0},  (\widetilde{\Xi}_{j', \omega})_\sharp\delta^Y_{\tilde{y}_0}\right)^p\\
    &=\sum_{j, j'\in \N}\chi_j(\omega)\tilde{\chi}_{j'}(\omega)\dist_Y(
    \gamma_j^{j'}(\omega)y_0, \tilde{y}_0)^p,
\end{align*}
which is bounded independent of $\omega\in \Omega$ due to assumption~\eqref{eqn: bounded orbits} and since $\{\chi_j\}_{j\in \N}$ is a partition of unity.
 Thus we see that $\mathcal{P}^\sigma_{p,q}(E)$ is well-defined.
\end{proof}

 \section{Proof of Theorem~\ref{thm: main disint}}\label{sec: proofs}

We give the proof of Theorem~\ref{thm: main disint}, divided into subsections.
Before embarking on the proof of Theorem~\ref{thm: main disint}, we recall some  characterizations of convergence with respect to $\mk_p^X$. 
\begin{theorem}[\cite{Villani09}*{Theorem~6.9}]\label{thm: wassconv}
Let $(X,\dist_X)$ be a complete, separable metric space and $1\leq p<\infty$. Then for a sequence $(\mu_\ell)_{\ell\in \mathbb{N}}$ in $\mathcal{P}_p(X)$
and $\mu \in \mathcal{P}_p(X)$, 
the following four conditions are equivalent.
\begin{itemize}
\setlength{\leftskip}{-15pt}
\item
$\displaystyle \lim_{\ell\to \infty} \mk_p^{X}(\mu_\ell,\mu)=0$.
\item
$(\mu_\ell)_{\ell\in \mathbb{N}}$ converges weakly to $\mu$ 
and 
\[
\lim_{\ell\to \infty} 
\int_{X} \dist_X(x_0, x)^p d\mu_\ell(x)=
\int_{X} \dist_X(x_0, x)^p d\mu(x)
\]
holds for some (hence all) $x_0\in X$.
\item
$(\mu_\ell)_{\ell\in \mathbb{N}}$ converges weakly to $\mu$ 
and 
\[
\lim_{r\to\infty}\varlimsup_{\ell\to \infty} 
\int_{X\setminus B_r^X(x_0)} \dist_X(x_0, x)^p d\mu_\ell(x)=0.
\]
\item
For any $\phi\in C(X)$ with 
$|\phi|\leq C(1+\dist_X(x_0,\cdot)^p)$ 
for some $C\in \mathbb{R}$ and $x_0\in X$, 
\[
\lim_{\ell\to \infty} 
\int_{X} \phi(x) d\mu_\ell(x)=
\int_{X} \phi(x) d\mu(x).
\]
\end{itemize}
\end{theorem}
\subsection{Complete, separable, metric}
We first prove that $(\mathcal{P}^\sigma_{p, q}(E), \scrmk_{p, q})$ is a complete metric space, and separable when $q<\infty$. It is easy to show $\scrmk_{p, q}$ is a metric, however completeness and separability will be more involved proofs, as there is no direct comparison between $\scrmk_{p, q}$ and the usual Monge--Kantorovich metrics (however, note Proposition~\ref{OT} below). Additionally, since our setting is on fiber bundles, $(\mathcal{P}^\sigma_{p, q}(E), \scrmk_{p, q})$ can not be identified with a metric space valued $L^q$ space, hence we must take a completely different approach.

Our proof of separability when $q<\infty$, 
is inspired by the arguments in~\cite{Varadarajan58}*{Theorem 1} and~\cite{HytonenvanNeervenVeraarWeis16}*{Remark 1.2.20}. 
\begin{remark}
We note that 
$\mathcal{P}^\sigma_{p, \infty}(E)$ is \emph{not} separable with respect to $\scrmk_{p, \infty}$ for any~$p$ if $Y$ is not a single point and $\sigma$ is such that there exists an uncountable family $\{\Omega_a\}_{a\in A}\subset \Omega$ of Borel sets in $\Omega$ so that $\sigma(\Omega_{a_1}\setminus \Omega_{a_2})>0$ for all distinct $a_1, a_2\in A$. Indeed, fix two distinct points $y_1$, $y_2\in Y$ and let 
\begin{align*}
\mathfrak{m}_a\coloneqq \left(\sum_{j\in \N}\mathds{1}_{V_j}(\mathds{1}_{\Omega_a}(\Xi_{j,\bullet})_\sharp\delta_{y_1}^Y+\mathds{1}_{\Omega\setminus \Omega_a}(\Xi_{j,\bullet})_\sharp\delta_{y_2}^Y)\right)\otimes \sigma
\quad 
\text{for }a\in A.  
\end{align*}
Then $\{\mathfrak{m}_a\}_{a\in A}$ is uncountable but 
\begin{align*}
\scrmk_{p, \infty}(\mathfrak{m}_{a_1}, \mathfrak{m}_{a_2})\geq\dist_Y(y_1, y_2)>0    
\end{align*}
 whenever $a_1\neq a_2$. As an example, if $E$ is a metric bundle whose base space~$\Omega$ is a Riemannian manifold and $\sigma$ is absolutely continuous with respect to the Riemannian volume, then for the sets $\Omega_a$ one can take geodesic balls of sufficiently small radius, centered at an uncountable collection of points.
\end{remark}
\begin{remark}
    As a consequence of the triangle inequality for $\scrmk_{p, q}$ below, we see if $\mathfrak{m}$, $\mathfrak{n}\in 
\mathcal{P}^\sigma_{p,1}(E)$, 
we have
\[
\mk_p^E(\mathfrak{m}^\omega, \mathfrak{n}^\omega) \in [0,\infty)
\quad
\text{for $\sigma$-a.e. $\omega$.}
\] 
Also a simple application of H\"{o}lder's inequality shows that
\[
\scrmk_{p, q}
\leq
\scrmk_{p', q'}, \
\mathcal{P}^\sigma_{p,q}(E)
\subset
\mathcal{P}^\sigma_{p',q'}(E)
\quad
\text{for }
p\leq p', q\leq q'.
\]
\end{remark}
We are now ready to prove the claims in Theorem~\ref{thm: main disint}~\eqref{thm: disint complete}.
\begin{proof}[Proof of Theorem~\ref{thm: main disint}~\eqref{thm: disint complete}]

\noindent
\textbf{(Metric):} 
Let $\mathfrak{m}$, $\mathfrak{n} \in \mathcal{P}^\sigma_{p, q}(E)$.
From the definition, it is immediate that 
\begin{align*}
\scrmk_{p, q}(\mathfrak{n}, \mathfrak{m})
=\scrmk_{p, q}(\mathfrak{m}, \mathfrak{n})
=\left\|\mk_p^E(\mathfrak{m}^{\bullet},\mathfrak{n}^{\bullet})\right\|_{L^q(\sigma)}
\geq 0,
\end{align*}
and equality holds if and only if 
$\mk_p^E(\mathfrak{m}^\bullet,\mathfrak{n}^\bullet)=0$, $\sigma$-a.e. 
Since $\mk_p^E$ is a metric when restricted to~$\mathcal{P}_p(\pi^{-1}(\{\omega\}))$ for each $\omega\in \Omega$, we see 
$\scrmk_{p, q}(\mathfrak{m}, \mathfrak{n})=0$ if and only if 
$\mathfrak{m}^\omega=\mathfrak{n}^\omega$ for $\sigma$-a.e. $\omega$,
that is, $\mathfrak{m}=\mathfrak{n}$ by~\nameref{thm: disintegration}.
Using the triangle inequality for $\mk_p^E$ together with Minkowski's inequality, we have for 
$\mathfrak{m}_1$, $\mathfrak{m}_2$, $\mathfrak{m}_3\in\mathcal{P}^\sigma_{p, q}(E)$,
\begin{align*}
 \scrmk_{p, q}(\mathfrak{m}_1, \mathfrak{m}_3)
 &=\left\| \mk_p^E(\mathfrak{m}^\bullet_1, \mathfrak{m}^\bullet_3)\right\|_{L^q(\sigma)}\\
 &\leq \left\| \mk_p^E(\mathfrak{m}^\bullet_1, \mathfrak{m}^\bullet_2)+\mk_p^E(\mathfrak{m}^\bullet_2, \mathfrak{m}^\bullet_3)\right\|_{L^q(\sigma)}\\
 &\leq \left\| \mk_p^E(\mathfrak{m}^\bullet_1, \mathfrak{m}^\bullet_2)\right\|_{L^q(\sigma)}+\left\|\mk_p^E(\mathfrak{m}^\bullet_2, \mathfrak{m}^\bullet_3)\right\|_{L^q(\sigma)}\\
 &=\scrmk_{p, q}(\mathfrak{m}_1, \mathfrak{m}_2)+\scrmk_{p, q}(\mathfrak{m}_2, \mathfrak{m}_3).
\end{align*}
By the above triangle inequality, we also see
\begin{align*}
    \scrmk_{p, q}(\mathfrak{m}, \mathfrak{n})&\leq \scrmk_{p, q}(\delta^\bullet_{E,y_0}\otimes\sigma, \mathfrak{m})+\scrmk_{p, q}(\delta^\bullet_{E,y_0}\otimes\sigma, \mathfrak{n})<\infty
\end{align*}
for all $\mathfrak{m}$, $\mathfrak{n} \in \mathcal{P}^\sigma_{p, q}(E)$.

\noindent
\textbf{(Separability):} 
Assume $q<\infty$.
 Let $\{\nu_m\}_{m\in \mathbb{N}}$ be a $\mk_p^Y$-dense subset in~$ \mathcal{P}_p(Y)$ (recall that $(\mathcal{P}_p(Y), \mk_p^Y)$ is separable). 
Since $(\Omega, \dist_\Omega)$ is separable, there exists a countable algebra $\mathcal{Q}\subset 2^\Omega$ of mutually disjoint sets which generates the Borel $\sigma$-algebra on $\Omega$. 
Given $I\in \N$ and a finite collection $\{Q_i\}_{i=1}^I\subset \mathcal{Q}$, by Lemma~\ref{lem: pushforward measurable} if we define 
\begin{align*}
    \left(\mathfrak{n}_{\{Q_i\}_{i=1}^I}^\bullet\otimes \sigma\right)(A)
    &\coloneqq\sum_{j\in \N}\int_\Omega\mathds{1}_{V_j}(\omega)( \Xi_{j,\omega})_\sharp\left(\sum_{i=1}^I \mathds{1}_{Q_i}(\omega)\nu_i+\mathds{1}_{\Omega\setminus \bigcup_{i=1}^{I}Q_i}(\omega)
    \delta_{y_0}^Y\right)(A)d\sigma(\omega),
\end{align*}
we see that $\mathfrak{n}_{\{Q_i\}_{i=1}^I}^\bullet\otimes \sigma\in  \mathcal{P}^\sigma_{p, q}(E)$.
Now we claim that
    \begin{align*}
        \mathcal{D}
        \coloneqq \left\{
          \mathfrak{n}_{\{Q_i\}_{i=1}^I}^\bullet\otimes \sigma
    \Biggm|
    \{Q_i\}_{i=1}^I\subset \mathcal{Q}
    \text{ for } I\in \N
    \right\}
    \end{align*}
    is $\scrmk_{p, q}$-dense in $\mathcal{P}^\sigma_{p, q}(E)$. 
    Since $\mathcal{D}$ is countable this will prove separability.

To this end, for $m\in \N$ and $\omega\in \Omega$, define
\[ 
\mathfrak{n}_m^\omega\coloneqq \sum_{j\in \N}\mathds{1}_{V_j}(\omega)(\Xi_{j, \omega})_\sharp \nu_m\in \mathcal{P}_p(E),
\]
supported on $\pi^{-1}(\{\omega\})$, by Lemma~\ref{lem: pushforward measurable}, 
for a fixed Borel $A\subset E$ the map $\omega\mapsto \mathfrak{n}^\omega_m(A)$ is Borel. 
Now fix $\mathfrak{m}=\mathfrak{m}^\bullet\otimes \sigma\in \mathcal{P}^\sigma_{p, q}(E)$, then we can define a function $f_m:\Omega  \to \R$ by
\begin{align*}
 f_m(\omega):
 =\mk_p^E(\mathfrak{n}_m^\omega, \mathfrak{m}^\omega),
\end{align*}
which is then Borel for each $m\in \mathbb{N}$ by~\cite{AmbrosioGigliSavare08}*{Lemma 12.4.7}; note that if $\omega\in V_j$ for some $j$, then 
\[
f_m(\omega)=\mk^Y_p(\nu_m, (\Xi_{j, \omega}^{-1})_\sharp \mathfrak{m}^\omega).
\]
For  $\ell$, $m\in \N$, 
define the Borel set
\[
\Omega_{\ell, m}\coloneqq f_m^{-1}([0, \ell^{-1}))\cap \left(\bigcap_{i=1}^{m-1} f_i^{-1}([\ell^{-1}, \infty))\right),
\]
note $\{\Omega_{\ell, m}\}_{m\in \mathbb{N}}$ is a cover of $\Omega$ consisting of mutually disjoint sets for each $\ell\in \N$.
Let us also write
\begin{align*}
    \widetilde{\delta}^\omega_{E, y_0}\coloneqq \sum_{j\in \N}\mathds{1}_{V_j}(\omega)(\Xi_{j, \omega})_\sharp\delta^Y_{y_0},
\end{align*}
again by Lemma~\ref{lem: pushforward measurable} the measure (whose disintegration with respect to $\pi$ is given by) $ \widetilde{\delta}^\omega_{E, y_0}\otimes \sigma$ belongs to $\mathcal{P}^\sigma_{p, q}(E)$. 
For each $\ell\in\mathbb{N}$,
since 
\begin{align*}
    \left\|\mk_p^E(\widetilde{\delta}^\bullet_{E, y_0}, \mathfrak{m}^\bullet)\right\|_{L^q(\sigma)}
\leq \scrmk_{p, q}(\widetilde{\delta}^\bullet_{E, y_0}\otimes\sigma, \delta^\bullet_{E, y_0}\otimes\sigma)+\scrmk_{p, q}(\delta^\bullet_{E, y_0}\otimes\sigma, \mathfrak{m})
<\infty,
\end{align*}
there exists $I_\ell\in \N$ such that
\begin{align*}
\left\|
\mk_p^E\left(
\widetilde{\delta}^\bullet_{E, y_0},
\mathfrak{m}^{\bullet}
\right)\mathds{1}_{\Omega\setminus  \cup_{i=1}^{I_\ell} \Omega_{\ell, i}}
\right\|_{L^q(\sigma)}
<\ell^{-1}.
\end{align*}
Now for $\omega\in \Omega$ and $\ell\in \N$, define the measures $\mathfrak{m}_\ell^\omega\in \mathcal{P}(E)$ by
\begin{align*}
    \mathfrak{m}_\ell^\omega:
    &=\sum_{j\in \N}\mathds{1}_{V_j}(\omega)(\Xi_{j, \omega})_\sharp\left(\sum_{i=1}^{I_\ell}\mathds{1}_{\Omega_{\ell, i}}(\omega)\nu_m\right)+\mathds{1}_{\Omega\setminus \bigcup_{i=1}^{I_\ell}\Omega_{\ell, i}}(\omega)
    \widetilde{\delta}^\omega_{E, y_0}.
\end{align*}
By Lemma~\ref{lem: pushforward measurable}, we have  $\mathfrak{m}_\ell\coloneqq \mathfrak{m}_\ell^\bullet\otimes \sigma\in \mathcal{P}^\sigma_{p, q}(E)$, and for any $\ell\in \N$ and $1\leq i\leq I_\ell$, we have $\mathfrak{m}_\ell^{\omega}= \mathfrak{n}_i^{\omega}$ whenever $\omega\in \Omega_{\ell, i}$.
Then by definition of $I_\ell$, 
\begin{align*}
\scrmk_{p, q}(\mathfrak{m}_\ell, \mathfrak{m})
&=
\left\|
\sum_{i=1}^{I_\ell}
\mk_p^E\left(\mathfrak{m}_\ell^{\bullet},\mathfrak{m}^{\bullet}\right)\mathds{1}_{\Omega_{\ell, i}}
+
\mk_p^E\left(
\mathfrak{m}_\ell^{\bullet},
\mathfrak{m}^{\bullet}
\right)\mathds{1}_{\Omega\setminus  \bigcup_{i=1}^{I_\ell} \Omega_{\ell, i}}
\right\|_{L^q(\sigma)}\\
&\leq 
\left\|
\sum_{i=1}^{I_\ell}\mk_p^E\left(\mathfrak{n}_i^{\bullet},\mathfrak{m}^{\bullet}
\right)
\mathds{1}_{\Omega_{\ell, i}}\right\|_{L^q(\sigma)}
+
\left\|
\mk_p^E\left(
\widetilde{\delta}^\bullet_{E, y_0},
\mathfrak{m}^{\bullet}
\right)\mathds{1}_{\Omega\setminus  \bigcup_{i=1}^{I_\ell} \Omega_{\ell, i}}
\right\|_{L^q(\sigma)}\\
&<
\left\|\ell^{-1}
\sum_{i=1}^{I_\ell}
\mathds{1}_{\Omega_{\ell, i}}\right\|_{L^q(\sigma)}
+
\ell^{-1}\\
&\leq 2\ell^{-1}.
\end{align*}

Fix $\varepsilon>0$, and let $\ell_0\in \N$ be such that 
\begin{align}\label{eqn: close to m}
    \scrmk_{p, q}(\mathfrak{m}_{\ell_0}, \mathfrak{m})<\varepsilon.
\end{align}
We now construct an element of~$\mathcal{D}$ approximating $\mathfrak{m}_{\ell_0}$. 
Let 
\begin{align*}
M\coloneqq \max_{1\leq i, i'\leq I_{\ell_0}}
\left\{
\max\left\{
\mk_p^Y(\nu_i, \nu_{i'})^q,\mk_p^Y(\delta_{y_0}^Y, \nu_{i'})^q\right\}\right\}.
\end{align*}
By \cite{HytonenvanNeervenVeraarWeis16}*{Lemma A.1.2}, for each $1\leq i\leq I_{\ell_0}$ there exists a set $\widetilde{Q}_i\in \mathcal{Q}$ with the property that 
$\sigma(\widetilde{Q}_i\Delta \Omega_{\ell_0, i})<\varepsilon^q/(M I_{\ell_0})$,
using these define 
\[
Q_1\coloneqq \widetilde{Q}_1, \qquad Q_i\coloneqq \widetilde{Q}_i\setminus \bigcup_{i'=1}^{i-1}Q_{i'} \quad
\text{for }2\leq i\leq I_{\ell_0}.
\]
We observe from Remark~\ref{rem: MK convex} that
\begin{align*}
\mk_p^E\left(\widetilde{\delta}^\omega_{E, y_0},
\sum_{j\in \N}\mathds{1}_{V_j}(\omega)(\Xi_{j, \omega})_\sharp\nu_i\right) 
&
\leq
\sum_{j\in \N}\mathds{1}_{V_j}(\omega)
\mk_p^E\left((\Xi_{j, \omega})_\sharp \delta^Y_{y_0},
(\Xi_{j, \omega})_\sharp\nu_i\right) \\
&=
\sum_{j\in \N}\mathds{1}_{V_j}(\omega)
\mk_p^Y\left(\delta^Y_{y_0}, \nu_i\right)
=\mk_p^Y\left(\delta^Y_{y_0}, \nu_i\right).
\end{align*}
Similarly, for each $1\leq i'\leq I_{\ell_0}$,
we have
\begin{align*}
\mk_p^E\left(\sum_{j\in \N}\mathds{1}_{V_j}(\omega)(\Xi_{j, \omega})_\sharp \nu_{i'},
\sum_{j\in \N}\mathds{1}_{V_j}(\omega)(\Xi_{j, \omega})_\sharp\nu_i\right)
&
\leq
\mk_p^Y\left(\nu_{i'}, \nu_i\right).
\end{align*}
Together, these imply that for each $1\leq i\leq I_{\ell_0}$,
\begin{align}
\begin{split}
\label{eqn: approximation each piece}
&\int_{Q_i}\mk_p^E\left(\mathfrak{m}_{\ell_0}^\bullet, \sum_{j\in \N}\mathds{1}_{V_j} (\Xi_{j,  \bullet})_\sharp\nu_i\right)^qd\sigma \\
&=
\sum_{i'=1}^{I_{\ell_0}}
\int_{Q_i\cap \Omega_{\ell_0, i'}} \mk_p^E\left(\sum_{j\in \N}\mathds{1}_{V_j} (\Xi_{j,  \bullet})_\sharp \nu_{i'},
\sum_{j\in \N}\mathds{1}_{V_j} (\Xi_{j,  \bullet})_\sharp\nu_{i}\right)^qd\sigma \\
&\quad+
\int_{Q_i \setminus \bigcup_{i'=1}^{I_\ell}\Omega_{\ell_0, i'}}\mk_p^E\left(    \widetilde{\delta}^\bullet_{E, y_0},
\sum_{j\in \N}\mathds{1}_{V_j} (\Xi_{j,  \bullet})_\sharp\nu_{i}\right)^qd\sigma \\
&\leq
\sum_{i'\neq i, 1\leq i' \leq I_{\ell_0}}
\int_{Q_i\cap \Omega_{\ell_0, i'}} \mk_p^Y(\nu_{i'},\nu_{i})^qd\sigma
+
\int_{ Q_i\setminus \cup_{i'=1}^{I_{\ell_0}} \Omega_{\ell_0, i'}}
\mk_p^Y\left(\delta^Y_{y_0}, \nu_i\right)^qd\sigma \\
&\leq
M \cdot \sigma(Q_i\setminus \Omega_{\ell_0,i})\\
&<  \frac{\varepsilon^q}{I_{\ell_0}}.
\end{split}
\end{align}
On the other hand, setting \[
\Omega'\coloneqq \left[
 \Omega\setminus \bigcup_{i=1}^{I_{\ell_0}} (\widetilde{Q}_i\cup \Omega_{\ell_0, i})
 \right],
\]
we can see that 
\begin{align*}
\Omega\setminus \bigcup_{i=1}^{I_{\ell_0}}Q_i
=
\Omega'
 \cup\left[ 
 \left(\bigcup_{i=1}^{I_{\ell_0}}\Omega_{\ell_0, i}\right)
 \setminus\left(\bigcup_{i=1}^{I_{\ell_0}}\widetilde{Q}_i\right)
 \right]
\subset
\Omega' \cup \left[\bigcup_{i=1}^{I_{\ell_0}}\left(\Omega_{\ell_0, i}\setminus \widetilde{Q}_i \right)\right].
\end{align*}
Since 
$\mathfrak{m}_{\ell_0}^\omega=\widetilde{\delta}^\omega_{E, y_0}$ for $\omega\in \Omega'$ we find 
\begin{align}
\begin{split}
&\int_{\Omega\setminus \bigcup_{i=1}^{I_{\ell_0}}Q_i}
\mk_p^E(\mathfrak{m}_{\ell_0}^\omega, \widetilde{\delta}^\omega_{E, y_0})^qd\sigma(\omega)\\
&\leq \int_{\Omega'}
\mk_p^E(\mathfrak{m}_{\ell_0}^\omega, \widetilde{\delta}^\omega_{E, y_0})^qd\sigma(\omega)
+
\sum_{i=1}^{I_{\ell_0}}
\int_{\Omega_{\ell_0,i}\setminus\widetilde{Q}_i}\mk_p^E(\mathfrak{m}_{\ell_0}^\omega, \widetilde{\delta}^\omega_{E, y_0})^qd\sigma(\omega)\\
&\leq \sum_{i=1}^{I_{\ell_0}}\int_{\Omega_{\ell_0, i}\setminus \widetilde{Q}_i}\left(\sum_{j\in\N}\mathds{1}_{V_j}(\omega)\mk_p^E((\Xi_{j, \omega})_\sharp\nu_i,(\Xi_{j, \omega})_\sharp\delta_{y_0}^Y)\right)^qd\sigma(\omega)\\
&\leq
\sum_{i=1}^{I_{\ell_0}}\mk_p^E(\nu_i,\delta_{y_0}^Y)^q\cdot\sigma(\Omega_{\ell_0, i}\setminus \widetilde{Q}_i)\\
&\leq M\cdot \sum_{i=1}^{I_{\ell_0}}\sigma(\Omega_{\ell_0, i}\Delta \widetilde{Q}_i)\\
&<\varepsilon^q.
\label{eqn: approximation complement}
\end{split}
\end{align}
Thus if we take 
\begin{align*}
\mathfrak{n}^\bullet\coloneqq&
\sum_{j\in \N}\mathds{1}_{V_j}( \Xi_{j,\bullet})_\sharp
\left(\sum_{i=1}^I \mathds{1}_{Q_i}\nu_i+\mathds{1}_{\Omega\setminus \bigcup_{i=1}^{I}Q_i}\delta_{y_0}^Y\right)
=    
\sum_{j\in \N}\mathds{1}_{V_j}\sum_{i=1}^{I_{\ell_0}}\mathds{1}_{Q_i}(\Xi_{j, \bullet})_\sharp\nu_i+\mathds{1}_{\Omega\setminus\bigcup_{i=1}^{I_{\ell_0}}Q_i}\widetilde{\delta}^\omega_{E, y_0},
\end{align*}
we find for $\mathfrak{n}\coloneqq \mathfrak{n}^\bullet\otimes \sigma\in \mathcal{D}$, using \eqref{eqn: close to m}, \eqref{eqn: approximation each piece}, and \eqref{eqn: approximation complement} that
\begin{align*}
    \scrmk_{p, q}(\mathfrak{n}, \mathfrak{m})&\leq \scrmk_{p, q}(\mathfrak{n}, \mathfrak{m}_{\ell_0})+\scrmk_{p, q}(\mathfrak{m}_{\ell_0}, \mathfrak{m})
    <\left(1+2^{\frac{1}{q}}\right)\varepsilon,
\end{align*}
finishing the proof of separability.

\noindent
\textbf{(Completeness):} 
Let $(\mathfrak{m}_\ell)_{\ell\in \N}$ be a Cauchy sequence in $(\mathcal{P}^\sigma_{p, q}(E), \scrmk_{p,q})$. 
Then there exists $\Omega_{p, q} \subset \Omega$ such that 
$\sigma(\Omega_{p, q})=1$ and $(\mathfrak{m}^\omega_\ell)_{\ell\in \N}$ is Cauchy in $\mk_p^E$ for any $\omega\in \Omega_{p, q}$. Indeed, if $q=\infty$, then the claim is trivial.
In the case $q<\infty$, for any $\varepsilon_1, \varepsilon_2>0$, there exists some $L\in \N$ such that whenever $\ell_1, \ell_2\geq L$, we have $\scrmk_{p, q}(\mathfrak{m}_{\ell_1}, \mathfrak{m}_{\ell_2})<\varepsilon_1\varepsilon_2$. 
It follows from Chebyshev's inequality that 
\begin{align*}
 \sigma\left(\{\omega\in \Omega\mid \mk_p^E(\mathfrak{m}^\omega_{\ell_1}, \mathfrak{m}^\omega_{\ell_2}) \geq \varepsilon_1 \}\right)
 &\leq\varepsilon_1^{-q}\int_{\Omega} 
 \mk_p^E(\mathfrak{m}^\omega_{\ell_1}, \mathfrak{m}^\omega_{\ell_2})^qd\sigma_{n-1}(\omega)
 = \varepsilon_1^{-q}\scrmk_{p, q}(\mathfrak{m}_{\ell_1}, \mathfrak{m}_{\ell_2})^q
 <\varepsilon_2^q,
\end{align*}
for $\ell_1, \ell_2\geq L$. 
Now we can take a subsequence of $(\mathfrak{m}_\ell)_{\ell\in\mathbb{N}}$ (not relabeled) such that for all~$\ell\in\mathbb{N}$,
\begin{align*}
\sigma(\left\{\omega\in \Omega\mid \mk_p^E(\mathfrak{m}^\omega_{\ell}, \mathfrak{m}^\omega_{\ell+1})\geq 2^{-\ell} \right\})\leq 2^{-\ell}.
\end{align*}
Setting 
\[
\Omega_{p, q}\coloneqq \Omega\setminus \left(\bigcap_{m=1}^\infty\bigcup_{\ell=m}^\infty \left\{\omega\in \Omega\Bigm| \mk_p^E(\mathfrak{m}^\omega_{m}, \mathfrak{m}^\omega_{m+1})\geq 2^{-\ell} \right\}\right),
\]
we have
\begin{align*}
 \sigma(\Omega_{p, q})= 1-\sigma\left(\bigcap_{m=1}^\infty\bigcup_{\ell=m}^\infty \left\{\omega\in \Omega\Bigm| \mk_p^E(\mathfrak{m}^\omega_{m}, \mathfrak{m}^\omega_{m+1})\geq 2^{-\ell} \right\}\right)=1
\end{align*}
by the Borel--Cantelli lemma, thus the sequence $(\mathfrak{m}^\omega_\ell)_{\ell\in \N}$ is Cauchy in $\mk_p^E$ whenever $\omega\in \Omega_{p, q}$. 

Since $\mk_p^E$ is complete on $\mathcal{P}_p(E)$, 
for every $\omega\in \Omega_{p, q}$, 
there is $\mathfrak{m}^\omega\in \mathcal{P}_p(E)$ such that 
\begin{equation}\label{2}
\lim_{\ell\to\infty}\mk_p^E(\mathfrak{m}_\ell^{\omega}, \mathfrak{m}^{\omega})=0.
\end{equation}
Then, for $\phi\in C_b(E)$,
it follows from Theorem~\ref{thm: wassconv} that 
\[
\int_E \phi(u)d\mathfrak{m}^\omega(u)
=\lim_{\ell\to\infty}
\int_E \phi(u)d\mathfrak{m}^\omega_\ell(u),
\]
which is a $\mathcal{B}_\sigma$-measurable function in $\omega$ by \nameref{thm: disintegration}.
For any open set $A\subset E$, the sequence $\{\min\{1, m\dist_E(\cdot, E\setminus A)\}\}_{m\in \N}\subset C_b(E)$ of nonnegative functions monotonically increases pointwise everywhere to $\mathds{1}_A$, hence by monotone convergence we see the map
$\omega\mapsto \mathfrak{m}^\omega(A)$
is  Borel for all open $A\subset E$. Thus defining the function $\mathfrak{m}$ on Borel sets $A\subset E$~by
\begin{align*}
    \mathfrak{m}(A)\coloneqq \int_\Omega \mathfrak{m}^\omega(A)d\sigma(\omega),
\end{align*}
  using Remark~\ref{V58} we see $\mathfrak{m}\in \mathcal{P}(E)$.
  Also for $\phi\in C_b(E)$ since  each~$\mathfrak{m}_\ell^\omega$ and~$\sigma$ are probability measures, the dominated convergence theorem yields 
\begin{equation*}
\int_E\phi d\mathfrak{m}=\int_\Omega
\int_E \phi(u)d\mathfrak{m}^\omega(u)
d\sigma(\omega)
=\lim_{\ell\to\infty}
\int_\Omega
\int_E \phi(u)d\mathfrak{m}_\ell^\omega(u)
d\sigma(\omega),
\end{equation*}
thus $\mathfrak{m}\in \mathcal{P}^\sigma(E)$; the uniqueness in \nameref{thm: disintegration} implies that $\mathfrak{m}=\mathfrak{m}^{\bullet} \otimes \sigma$.

Now fix $\varepsilon>0$, then there exists $\ell_0$ such that for all $\ell$, $m\geq \ell_0$ we have $\scrmk_{p, q}(\mathfrak{m}_m, \mathfrak{m}_\ell)<\varepsilon$. Then using Fatou's lemma when $q<\infty$ and directly by definition for $q=\infty$, and recalling \eqref{2},
\begin{align}\label{22}
\begin{split}
 \left\| \mk_p^E(\mathfrak{m}_\ell^\bullet, \mathfrak{m}^\bullet)\right\|_{L^q(\sigma)}
 &= \left\| \varliminf_{m\to\infty}\mk_p^E(\mathfrak{m}_\ell^\bullet, \mathfrak{m}_m^\bullet)\right\|_{L^q(\sigma)}
 \leq\varliminf_{m\to\infty} \left\| \mk_p^E(\mathfrak{m}_\ell^\bullet, \mathfrak{m}_m^\bullet)\right\|_{L^q(\sigma)}
 <\varepsilon,
\end{split}
\end{align}
which ensures $ \mk_p^E(\mathfrak{m}_\ell^\bullet, \mathfrak{m}^\bullet)\in L^q(\sigma)$.
Since we have
\[
\mk_p^E(\delta^\omega_{E,y_0}, \mathfrak{m}^\omega)
\leq
\mk_p^E(\delta^\omega_{E,y_0}, \mathfrak{m}_{\ell_0}^\omega)+
\mk_p^E(\mathfrak{m}_{\ell_0}^\omega, \mathfrak{m}^\omega)
\quad
\text{for }\omega \in \Omega_{p, q},
\]
 $\sigma(\Omega_{p, q})=1$, and $\mathfrak{m}_{\ell_0}\in \mathcal{P}^\sigma_{p, q}(E)$, we conclude $\mathfrak{m}\in \mathcal{P}^\sigma_{p,q}(E)$.
It also follows from \eqref{22} that 
\[
\lim_{\ell\to\infty} 
\scrmk_{p,q}(\mathfrak{m}_\ell, \mathfrak{m})
 =\lim_{\ell\to\infty}
 \left\| \mk_p^E(\mathfrak{m}_\ell^\bullet,\mathfrak{m}^\bullet)\right\|_{L^q(\sigma)}
=0
\]
for the particular chosen subsequence. Since the original sequence is Cauchy, the full sequence also converges in $\scrmk_{p, q}$ to $\mathfrak{m}$. 
This proves completeness.
\end{proof}

\subsection{Existence of geodesics}
We now prove that $(\mathcal{P}^\sigma_{p, q}(E), \scrmk_{p, q})$ is a geodesic space. 
When $p>1$ on a geodesic space $Y$, a minimal geodesic in $(\mathcal{P}_p(Y),\mk_p^Y)$ can be obtained as a family of pushforwards of what is known as a \emph{dynamic optimal coupling}. More specifically, we start by recalling the following space (which will also be used in the proof of Theorem~\ref{thm: main disint}~\eqref{thm: disint geodesic}). 
\begin{definition}\label{def: geod}
    Suppose $(Z, \dist_Z)$ is complete, separable, and a geodesic space. We denote by $\mathcal{G}(Z)$ the space of minimal geodesics $\rho:[0,1]\to Z$ with respect to $\dist_Z$. Define the metric $\dist_{\mathcal{G}(Z)}$ on~$\mathcal{G}(Z)$ by 
    \begin{align*}
        \dist_{\mathcal{G}(Z)}(\rho_1, \rho_2)\coloneqq\sup_{\tau\in [0, 1]}\dist_Z(\rho_1(\tau), \rho_2(\tau)).
    \end{align*}
For $\tau\in [0, 1]$ the evaluation map $\mathrm{e}^\tau: \mathcal{G}(Z)\to Z$ is defined by $\mathrm{e}^\tau(\rho)\coloneqq \rho(\tau)$.
\end{definition}
We can see that $(\mathcal{G}(Z), \dist_{\mathcal{G}(Z)})$ is complete and separable since it is a closed subset of $C([0, 1]; Z)$ with the same metric $\dist_{\mathcal{G}(Z)}$, which is also complete and separable by \cite{Srivastava98}*{Theorem 2.4.3}.
Then it is known that $\mk^Z_p$ minimal geodesics have the following description.
\begin{proposition}[\cite{Villani09}*{Corollaries~7.22, 7.23,  and Theorem~7.30~(i)}]\label{prop: geodesics}
Let 
$(Z,\dist_Z)$ be a complete, separable geodesic space and $p>1$.
Then, for $\mu_0,\mu_1\in \mathcal{P}_p(Z)$,
there exists $\Gamma\in \mathcal{P}(\mathcal{G}(Z))$ such that 
$(\mathrm{e}^{0}\times \mathrm{e}^{1})_\sharp \Gamma$ is an $p$-optimal coupling between $\mu_0$ and $\mu_1$, and 
\[
\mathrm{e}^{\bullet}_\sharp \Gamma:[0,1]\to \mathcal{P}(Z)
\]
is a minimal geodesic from $\mu_0$ and $\mu_1$ in $(\mathcal{P}_p(Z),\mk^Z_p)$.
Moreover, for $\tau_1,\tau_2\in [0,1]$ the measure
$(\mathrm{e}^{\tau_1}\times \mathrm{e}^{\tau_2})_\sharp\Gamma \in \Pi(\mathrm{e}^{\tau_1}_\sharp \Gamma,\mathrm{e}^{\tau_2}_\sharp \Gamma)$ is a $p$-optimal coupling. Conversely, for any $\Gamma\in \mathcal{P}(\mathcal{G}(Z))$ such that $(\mathrm{e}^0\times \mathrm{e}^0)_\sharp\Gamma$ is a $p$-optimal coupling between $ \mathrm{e}^0_\sharp \Gamma$ and $\mathrm{e}^1_\sharp \Gamma$, 
\[
\mathrm{e}^{\bullet}_\sharp \Gamma:[0,1]\to \mathcal{P}(Z)
\]
is a minimal geodesic from $\mu_0$ and $\mu_1$ in $(\mathcal{P}_p(Z),\mk^Z_p)$.
\end{proposition}

We will take $\mk_p^Y$ minimal geodesics connecting each pair $\mathfrak{m}_1^\omega$ and $\mathfrak{m}_2^\omega$, then use these to construct a minimal geodesic for $\scrmk_{p, q}$. However, in order to do so we must make sure the dependence on $\omega$ is $\mathcal{B}_\sigma$-measurable, hence we will have to use the Kuratowski and Ryll-Nardzewski measurable selection theorem which we will now recall.
\begin{definition}
Let $(X,\mathcal{F}_X)$ be a measurable space and $(Z,\dist_Z)$ be a metric space.
A set-valued function $F$ from $X$ to $2^Z$ is said to be 
\emph{$\mathcal{F}_X$-weakly measurable} if
\[
\{ x\in X\ |\ F(x) \cap O \neq \emptyset \}\in\mathcal{F}_X
\]
for any open $O\subset Z$.
\end{definition}
\begin{remark}\label{rem: measurability alternate}
By \cite{KuratowskiRyll-Nardzewski65}*{Corollary 1}, it is equivalent to replace ``open'' by ``closed'' in the above definition; it is then clear that if $Z$ is $\sigma$-compact then it is also equivalent to replace ``open'' by ``compact''.
\end{remark}
\begin{theorem}[\cite{KuratowskiRyll-Nardzewski65}*{Main Theorem}]\label{measurableselection}
Let $(X,\mathcal{F}_X, \mu)$ be a measure space and 
$(Z,\dist_Z)$ a complete, separable metric space.
For a map $F:X\to 2^{Z}$,
if $F(x)$ is nonempty and closed for $\mu$-a.e. $x\in X$,
and 
$F$ is $\mathcal{F}_X$-weakly measurable, 
then there exists an $\mathcal{F}_X$-measurable map $f_\bullet:X\to Z$ such that $f_x\in F(x)$ for $\mu$-a.e. $x\in X$.
Such a map is called a \emph{measurable selection} of~$F$.
\end{theorem}
We now show a preliminary lemma on convergence of dynamic optimal couplings and their pushforwards.
\begin{lemma}\label{lem: pushforward continuity}
Let $(Z,\dist_Z)$ be a complete, separable, and geodesic space.
Then for any $\tau\in [0,1]$, the map $\mathrm{e}^\tau_\sharp:\mathcal{P}(\mathcal{G}(Z))\to \mathcal{P}(Z)$ is both weakly and $\mk_p^{\mathcal{G}(Z)}$-to-$\mk_p^Z$ continuous.
In particular, if $(\Gamma_\ell)_{\ell\in\N}$ converges to $\Gamma$ with respect to $\mk_p^{\mathcal{G}(Z)}$, the sequence 
$(\mathrm{e}^\tau_\sharp \Gamma_\ell)_{\ell\in \mathbb{N}}$ converges to $\mathrm{e}^\tau_\sharp \Gamma$ with respect to~$\mk_p^Z$.
\end{lemma}
\begin{proof}
Let $(\Gamma_\ell)_{\ell\in\mathbb{N}}$ be a weakly convergent sequence in $\mathcal{P}(\mathcal{G}(Z))$ with limit $\Gamma$.
For $\phi\in C_b(Z)$, we have $\phi \circ \mathrm{e}^\tau \in C_b(\mathcal{G}(Z))$ and 
\begin{align*}
\lim_{\ell\to \infty} \int_Z \phi(t) d \mathrm{e}^\tau_\sharp \Gamma_\ell(t)
&=
\lim_{\ell\to \infty} \int_{\mathcal{G}(Z)} \phi(  \mathrm{e}^\tau(\rho)) d \Gamma_\ell(\rho)
=
\int_{\mathcal{G}(Z)} \phi(  \mathrm{e}^\tau(\rho)) d \Gamma(\rho)
=\int_Z \phi(t) d \mathrm{e}^\tau_\sharp \Gamma(t),
\end{align*}
which shows weak continuity of $\mathrm{e}^\tau_\sharp$. Now if $(\Gamma_\ell)_{\ell\in \N}$ converges to $\Gamma$ in $\mk_p^{\mathcal{G}(Z)}$, the above implies $(\mathrm{e}^\tau_\sharp \Gamma_\ell)_{\ell\in \N}$ converges weakly to $\mathrm{e}^\tau_\sharp \Gamma$. Then if $\rho_0\in \mathcal{G}(Z)$ is identically $z_0\in Z$, by Theorem~\ref{thm: wassconv}
\begin{align*}
    \varlimsup_{\ell\to\infty}\int_{Z\setminus B_r^Z(z_0)}\dist_Z(z_0, z)^pd\mathrm{e}^\tau_\sharp\Gamma_\ell(z)
&=\varlimsup_{\ell\to\infty}\int_{\mathcal{G}(Z)}\dist_Z(z_0, \rho(\tau))^p\mathds{1}_{Z\setminus B_r^Z(z_0)}(\rho(\tau))d\Gamma_\ell(\rho)\\
    &\leq \varlimsup_{\ell\to\infty}\int_{\mathcal{G}(Z)\setminus B_r^{\mathcal{G}(Z)}(\rho_0)}\dist_{\mathcal{G}(Z)}(\rho_0, \rho)^pd\Gamma_\ell(\rho)
    \xrightarrow{r\to\infty}0,
\end{align*}
hence by another application of Theorem~\ref{thm: wassconv} we see $(\mathrm{e}^\tau_\sharp \Gamma_\ell)_{\ell\in \mathbb{N}}$ converges to $\mathrm{e}^\tau_\sharp \Gamma$ in $\mk_p^Z$.
\end{proof}
We are now ready to prove Theorem~\ref{thm: main disint}~\eqref{thm: disint geodesic}.
\begin{proof}[Proof of Theorem~\ref{thm: main disint}~\eqref{thm: disint geodesic}]
Recall we assume that $(Y, \dist_Y)$ is a geodesic space that is ball convex with respect to some $y_0\in Y$. If $p=1$, it is easy to see that $((1-\tau)\mathfrak{m}_0+\tau\mathfrak{m}_1)_{\tau\in [0,1]}$ is a minimal geodesic with respect to $\scrmk_{1, q}$ for any $1\leq q\leq \infty$ (see for example~\cite{KitagawaTakatsu24a}*{Lemma 2.10} (the result there is on $\mathcal{P}_1(\R^n)$, but the exact same proof holds for general $Y$) thus we assume $p>1$.

As previously mentioned, $(\mathcal{P}_p(Y^2), \mk_p^{Y^2})$ is a complete, separable metric space.
For $t$, $s\in Y$, 
since we have
\[
\dist_Y(t, s)^p
=\left(\dist_Y(t, s)^2\right)^{\frac{p}{2}}
\leq 2^{\frac{p}{2}}(\dist_{y_0}(t)^2+\dist_{y_0}(s)^2)^{\frac{p}{2}} 
=2^{\frac{p}{2}}\dist_{Y^2}((y_0, y_0), (t,s))^p,
\]
Theorem~\ref{thm: wassconv} yields that 
the function $\mathcal{C}(\gamma)\coloneqq \left\|\dist_Y^p\right\|_{L^1(\gamma)}$ on $(\mathcal{P}_p(Y^2),\mk_p^{Y^2})$
is continuous.

Now there is a set   $\Omega'\in \mathcal{B}_\sigma$ of full $\sigma$ measure so that $\mathfrak{m}_0^\omega$, $\mathfrak{m}_1^\omega\in\mathcal{P}_p(\pi^{-1}(\{\omega\}))$ for all $\omega\in \Omega'$.
For $i=1$, $2$, let us write
\begin{align*}
    \mu_i^\omega\coloneqq \sum_{j\in \N}\chi_j(\omega)(\Xi^{-1}_{j, \omega})_\sharp\mathfrak{m}_i^\omega
\end{align*} 
which belongs to $\mathcal{P}_p(Y)$ for $\omega\in \Omega'$. 
Now define 
$F:\Omega\to 2^{\mathcal{P}_p(\mathcal{G}(Y))}$ by 
\begin{align*}
    F(\omega)\coloneqq \left\{\Gamma\in \mathcal{P}_p(\mathcal{G}(Y))
\bigm|
\text{$\mathrm{e}^\bullet_\sharp \Gamma$ is an $\mk_p^Y$ minimal geodesic
from $\mu_0^\omega$ to $\mu_1^\omega$}
\right\};
\end{align*}
note that if $\Gamma\in F(\omega)$
then $(\mathrm{e}^0\times\mathrm{e}^1)_\sharp\Gamma\in \Pi(\mu_0^\omega, \mu_1^\omega)$ is a $p$-optimal coupling 
by~\cite{Villani09}*{Corollary 7.22}.

We now show that $F$ satisfies the hypotheses of the Kuratowski and Ryll-Nardzewski selection theorem, Theorem~\ref{measurableselection}.

\noindent
{\bf Claim~1.} The set $F(\omega)$ is nonempty and closed for $\sigma$-a.e. $\omega$.\\
{\it Proof of Claim~$1$.}
By Proposition~\ref{prop: geodesics}, for any $\omega \in \Omega'$ there is a $\Gamma\in \mathcal{P}(\mathcal{G}(Y))$ such that $\mathrm{e}^\bullet_\sharp \Gamma$ is a minimal geodesic from $\mu_0^\omega$ to $\mu_1^\omega$. 
Additionally, if $\rho_0\in \mathcal{G}(Y)$ is 
identically equal to $y_0$, since Proposition~\ref{prop: geodesics} also yields that $(\mathrm{e}^0\times \mathrm{e}^1)_\sharp\Gamma$ is a $p$-optimal coupling between $\mu_0^\omega$ and $\mu_1^\omega$, we have 
\begin{align*}
    \int_{\mathcal{G}(Y)} \dist_{\mathcal{G}(Y)}(\rho,\rho_0)^pd\Gamma(\rho)
&=  \int_{\mathcal{G}(Y)} \left(\sup_{\tau\in [0, 1]} \dist_Y( \rho(\tau), \rho_0(\tau)) \right)^pd\Gamma(\rho)\\    
&\leq 2^{p-1}\int_{\mathcal{G}(Y)} \sup_{\tau\in [0, 1]}\left(\dist_Y(\rho(0),y_0)^p+\dist_Y(\rho(0), \rho(\tau))^p\right)d\Gamma(\rho)\\
    &=2^{p-1}\int_{\mathcal{G}(Y)} \sup_{\tau\in [0, 1]}\left(\dist_{y_0}(\rho(0))^p+\tau^p \dist_Y(\rho(0), \rho(1))^p\right)d\Gamma(\rho)\\
    &= 2^{p-1}\int_Y\dist_{y_0}(t)^pd\mathrm{e}^0_\sharp\Gamma(t)+2^{p-1}\int_{Y^2}\dist_Y(t, s)^pd(\mathrm{e}^0\times \mathrm{e}^1)_\sharp\Gamma(t, s)\\
    &=2^{p-1}\int_Y\dist_{y_0}(t)^pd\mu_0^\omega(t)+2^{p-1}\mk_p^Y(\mu_0^\omega, \mu_1^\omega)\\
    &<\infty,
\end{align*}
hence $\Gamma\in \mathcal{P}_p(\mathcal{G}(Y))$, thus we have $F(\omega)\neq \emptyset$.
Now given $\omega\in \Omega'$, 
if $(\Gamma_\ell)_{\ell\in\mathbb{N}}\subset F(\omega)$ converges 
in $(\mathcal{P}_p(\mathcal{G}(Y)),\mk_p^{\mathcal{G}(Y)})$, by Lemma~\ref{lem: pushforward continuity} 
the sequence $(\mathrm{e}^\tau_\sharp\Gamma_\ell)_{\ell\in\mathbb{N}}$ converges to $\mathrm{e}^\tau_\sharp\Gamma$ in $\mk^Y_p$ for each $\tau\in [0, 1]$. Thus for $\tau_1$, $\tau_2\in [0, 1]$ we have
\begin{align*}
    \mk_p^Y(\mathrm{e}^{\tau_1}_\sharp \Gamma, \mathrm{e}^{\tau_2}_\sharp \Gamma)
    &=\lim_{\ell\to\infty} \mk_p^Y(\mathrm{e}^{\tau_1}_\sharp \Gamma_\ell, \mathrm{e}^{\tau_2}_\sharp \Gamma_\ell)\\
    &=\lim_{\ell\to\infty} \lvert \tau_1-\tau_2\rvert\mk_p^Y(\mathrm{e}^0_\sharp \Gamma_\ell, \mathrm{e}^1_\sharp \Gamma_\ell)
    =\lvert \tau_1-\tau_2\rvert\mk_p^Y(\mathrm{e}^0_\sharp \Gamma, \mathrm{e}^1_\sharp \Gamma),
\end{align*}
hence $\Gamma\in F(\omega)$; in other words $F(\omega)$ is closed in 
$(\mathcal{P}_p(\mathcal{G}(Y)),  \mk_p^{\mathcal{G}(Y)})$.
\hfill $\diamondsuit$

\noindent
{\bf Claim~2.} $F$ is $\mathcal{B}_\sigma$-weakly measurable.\\
{\it Proof of Claim~$2$.}
For $\Gamma \in \mathcal{P}_p(\mathcal{G}(Y))$, 
define $\Phi_\Gamma: \Omega'\to \mathbb{R}^3$ by 
\begin{align*}
\Phi_\Gamma(\omega)
&\coloneqq\left(
\mk_p^Y\left(\mathrm{e}^0_\sharp\Gamma, \mu_0^\omega\right)^p,
\mk_p^Y\left(\mathrm{e}^1_\sharp\Gamma, \mu_1^\omega\right)^p,
\left| \mathcal{C}\left((\mathrm{e}^0\times \mathrm{e}^1)_\sharp\Gamma\right)-\mk_p^Y(\mu_0^\omega, \mu_1^\omega)^p\right|
\right).
\end{align*}
We see $\Phi_\Gamma$ is $\mathcal{B}_\sigma$-measurable by combining Lemma~\ref{lem: pushforward measurable} and \cite{AmbrosioGigliSavare08}*{Lemma 12.4.7}. 
Since $(\mathcal{G}(Y), \dist_{\mathcal{G}(Y)})$ is complete and separable, the space  $(\mathcal{P}_p(\mathcal{G}(Y)), \mk_p^{\mathcal{G}(Y)})$ is complete and separable. Fix a closed set $K$ in $(\mathcal{P}_p(\mathcal{G}(Y)), \mk_p^{\mathcal{G}(Y)})$, then
there exists a countable set $\{\Gamma_\ell\}_{\ell\in \mathbb{N}}$ that is $\mk_p^{\mathcal{G}(Y)}$-dense in~$K$. 
Set 
\begin{align*}
B&\coloneqq \bigcap_{\widetilde{m}=1}^\infty 
\bigcup_{\ell=1}^\infty 
\Phi_{\Gamma_\ell}^{-1}\left(\left[0,\widetilde{m}^{-1}\right)^3\right),\qquad
\Omega_K\coloneqq \{ \omega \in \Omega'\mid F(\omega) \cap K \neq \emptyset  \},
\end{align*}
by the $\mathcal{B}_\sigma$-measurability of each $\Phi_{\Gamma_\ell}$, 
we find $B\in \mathcal{B}_\sigma$. 
We will now show that $\Omega_K=B$.

If $\omega\in \Omega_K$, there exists $\Gamma\in F(\omega) \cap K$,
and a sequence $(\Gamma_{\ell_m})_{m\in \N}$ taken from $(\Gamma_\ell)_{\ell\in \mathbb{N}}$ that converges to $\Gamma$ with respect to $\mk_p^{\mathcal{G}(Y)}$. Then by Lemma~\ref{lem: pushforward continuity}, the sequence $(\mathrm{e}^i_\sharp \Gamma_{\ell_m})_{m\in \N}$ converges in $\mk_p^Y$ to $\mu_i^\omega=\mathrm{e}^i_\sharp \Gamma$, for $i=0$, $1$.
Similarly, the convergence of $(\Gamma_{\ell_m})_{m \in \mathbb{N}}$ to $\Gamma$ in $\mk_p^{\mathcal{G}(Y)}$ implies convergence of $((\mathrm{e}^0\times \mathrm{e}^1)_\sharp \Gamma_{\ell_m})_{m \in \mathbb{N}}$ to $(\mathrm{e}^0\times \mathrm{e}^1)_\sharp \Gamma$ in $\mk_p^{Y^2}$, hence the continuity of $\mathcal{C}$ implies that
\begin{align*}
\lim_{m\to\infty}\left| \mathcal{C}((\mathrm{e}^0\times \mathrm{e}^1)_\sharp\Gamma_{\ell_m})-\mk_p^Y\left(\mu_0^\omega, \mu_1^\omega\right)^p\right|
&=\lim_{m\to\infty}\left|\mathcal{C}((\mathrm{e}^0\times \mathrm{e}^1)_\sharp\Gamma_{\ell_m})-\mathcal{C}((\mathrm{e}^0\times \mathrm{e}^1)_\sharp\Gamma)\right|
=0.
\end{align*}
Thus for any $\widetilde{m}\in \N$, if $m$ is sufficiently large, we have 
$\Phi_{\Gamma_{\ell_m}}(\omega)\in [0,\widetilde{m}^{-1})^3$
which yields $\omega \in B$.

Now assume $\omega\in B$. 
For each $\widetilde{m}\in \mathbb{N}$,
there is $\ell(\widetilde{m})\in \mathbb{N}$ such that 
$\Phi_{\Gamma_{\ell(\widetilde{m})}}(\omega)\in [0,\widetilde{m}^{-1})^3$, that is,
\begin{align}
\begin{split}\label{F}
&\mk_p^Y(\mathrm{e}^0_\sharp\Gamma_{\ell(\widetilde{m})}, \mu_0^\omega)^p<\widetilde{m}^{-1},\qquad
\mk_p^Y(\mathrm{e}^1_\sharp\Gamma_{\ell(\widetilde{m})}, \mu_1^\omega)^p<\widetilde{m}^{-1}, \\
&\left|\mathcal{C}((\mathrm{e}^0\times \mathrm{e}^1)_\sharp\Gamma_{\ell(\widetilde{m})})-\mk_p^Y(\mu_0^\omega, \mu_1^\omega)^p\right|<\widetilde{m}^{-1}. 
\end{split}
\end{align}
As the sets $\{\mathrm{e}^0_\sharp\Gamma_{\ell(\widetilde{m})}\}_{\widetilde{m}\in \N}\cup\{ \mu_0^\omega\}$ and $\{\mathrm{e}^1_\sharp\Gamma_{\ell(\widetilde{m})}\}_{\widetilde{m}\in \N}\cup\{ \mu_1^\omega\}$ are compact in $(\mathcal{P}_p(Y), \mk_p^Y)$, by \cite{Villani09}*{Corollary 7.22} there exists a subsequence of $(\Gamma_{\ell(\widetilde{m})})_{\widetilde{m}\in \N}$ (not relabeled) that converges weakly to some $\Gamma\in\mathcal{P}(\mathcal{G}(Y))$. 
Since $(Y, \dist_Y)$ is ball convex with respect to $y_0$, and $\rho_0\in \mathcal{G}(Y)$ is identically~$y_0$,
\begin{align*}
&\varlimsup_{\widetilde{m}\to\infty}\int_{\mathcal{G}(Y)\setminus B_r^{\mathcal{G}(Y)}(\rho_0)}\dist_{\mathcal{G}(Y)}(\rho, \rho_0)^pd\Gamma_{\ell(\widetilde{m})}(\rho)\\
&\leq 
\varlimsup_{\widetilde{m}\to\infty}\int_{\{\rho\in \mathcal{G}(Y)\mid \max_{i=0,1}\dist_{y_0}(\rho(i))\geq r\}}\left(\max_{i=0,1}\dist_{y_0}\rho(i))\right)^pd\Gamma_{\ell(\widetilde{m})}(\rho)\\
&\leq \varlimsup_{\widetilde{m}\to\infty}\int_{\{\rho\in \mathcal{G}(Y)\mid \dist_{y_0}(\rho(0))\geq r\}}\dist_{y_0}(\rho(0))^pd\Gamma_{\ell(\widetilde{m})}(\rho)+\varlimsup_{\widetilde{m}\to\infty}\int_{\{\rho\in \mathcal{G}(Y)\mid \dist_{y_0}(\rho(1))\geq r\}}\dist_{y_0}(\rho(1))^pd\Gamma_{\ell(\widetilde{m})}(\rho)\\
&= \varlimsup_{\widetilde{m}\to\infty}\int_{Y\setminus B_r^Y(y_0)}\dist_{y_0}(t)^pd\mathrm{e}^0_\sharp\Gamma_{\ell(\widetilde{m})}(t)+\varlimsup_{\widetilde{m}\to\infty}\int_{Y\setminus B_r^Y(y_0)}\dist_{y_0}(t)^pd\mathrm{e}^1_\sharp\Gamma_{\ell(\widetilde{m})}(t)
\xrightarrow{r\to \infty} 0
\end{align*}
by~\eqref{F} and Theorem~\ref{thm: wassconv}, hence $\Gamma_{\ell(\widetilde{m})}\to \Gamma$ in $\mk_p^{\mathcal{G}(Y)}$ as $\widetilde{m}\to\infty$. Since $K$ is $\mk_p^{\mathcal{G}(Y)}$-closed, this implies $\Gamma\in K$. From~\eqref{F} we see $(\mathrm{e}^0\times \mathrm{e}^1)_\sharp\Gamma$ is a $p$-optimal coupling between $\mu_0^\omega$ and $\mu_1^\omega$, hence from Proposition~\ref{prop: geodesics} we have that $\Gamma\in F(\omega)$. Thus $\omega \in \Omega_K$, proving $\Omega_K=B\in \mathcal{B}_\sigma$, and in particular $F$ is $\mathcal{B}_\sigma$-weakly measurable.
\hfill $\diamondsuit$

As mentioned previously  $(\mathcal{P}_p(\mathcal{G}(Y)), \mk_p^{\mathcal{G}(Y)})$ is complete and separable, hence we can apply Theorem~\ref{measurableselection}, to find a $\mathcal{B}_\sigma$-measurable selection $\Gamma_\bullet:\Omega\to \mathcal{P}_p(\mathcal{G}(Y))$ of~$F$, defined $\sigma$-a.e.  
By Lemma~\ref{lem: pushforward continuity}, as the composition of a continuous map $e^\tau_\sharp$ with an $\mathcal{B}_\sigma$-measurable map $\Gamma_\bullet$, the map 
$\mathrm{e}^\tau_\sharp \Gamma_\bullet: \Omega\to \mathcal{P}_p(Y)$ is $\mathcal{B}_\sigma$-measurable for each $\tau\in[0, 1]$.

Thus we can argue again as in Remark~\ref{V58} to see the linear functional
\begin{align*}
    \mathfrak{m}_\tau(A)\coloneqq \int_\Omega \left(\sum_{j\in\N}\chi_j(\omega) (\Xi_{j, \omega})_\sharp\mathrm{e}^\tau_\sharp \Gamma_\omega(A)\right) d\sigma(\omega)
\end{align*}
is a nonnegative probability measure on $E$, and whose disintegration satisfies $\sigma$-a.e.,
\begin{align*}
    \mathfrak{m}_\tau^\bullet=\sum_{j\in\N}\chi_j(\Xi_{j, \bullet})_\sharp\mathrm{e}^\tau_\sharp \Gamma_\bullet.
\end{align*} 
Now fix $0\leq \tau_1<\tau_2\leq 1$. 
By the construction of $\Gamma_\bullet$,  
\begin{align*}
    \scrmk_{p, q}(\mathfrak{m}_{\tau_1}, \mathfrak{m}_{\tau_2})
    &=\left\| \mk_p^E\left(\sum_{j\in\N}\chi_j(\Xi_{j, \bullet})_\sharp\mathrm{e}^{\tau_1}_\sharp \Gamma_\bullet, \sum_{j'\in\N}\chi_{j'}(\Xi_{j', \bullet})_\sharp\mathrm{e}^{\tau_2}_\sharp \Gamma_\bullet\right)\right\|_{L^q(\sigma)}\\
    &\leq \left\| \sum_{j\in\N}\chi_j\mk_p^E((\Xi_{j, \bullet})_\sharp\mathrm{e}^{\tau_1}_\sharp \Gamma_\bullet, (\Xi_{j, \bullet})_\sharp\mathrm{e}^{\tau_2}_\sharp \Gamma_\bullet)\right\|_{L^q(\sigma)}
    =\left\| \sum_{j\in\N}\chi_j\mk_p^Y(\mathrm{e}^{\tau_1}_\sharp \Gamma_\bullet, \mathrm{e}^{\tau_2}_\sharp \Gamma_\bullet)\right\|_{L^q(\sigma)}\\
    &=\left| \tau_1-\tau_2\right|\left\| \mk_p^Y\left(\sum_{j\in\N}\chi_j(\Xi^{-1}_{j, \bullet})_\sharp\mathfrak{m}_0^\bullet, \sum_{j'\in\N}\chi_{j'}(\Xi^{-1}_{j', \bullet})_\sharp\mathfrak{m}_1^\bullet \right)\right\|_{L^q(\sigma)}\\
    &\leq\left| \tau_1-\tau_2\right|\left\| \sum_{j\in\N}\chi_j\mk_p^Y((\Xi^{-1}_{j, \bullet})_\sharp\mathfrak{m}_0^\bullet, (\Xi^{-1}_{j, \bullet})_\sharp\mathfrak{m}_1^\bullet )\right\|_{L^q(\sigma)}\\\
    &=\left| \tau_1-\tau_2\right|\left\| \mk_p^E(\mathfrak{m}_0^\bullet, \mathfrak{m}_1^\bullet)\right\|_{L^q(\sigma)}
    =\left| \tau_1-\tau_2\right| \scrmk_{p, q}(\mathfrak{m}_0, \mathfrak{m}_1).
\end{align*}
Due to the triangle inequality, the above is enough to conclude equality holds in Definition~\ref{def: geodineq} of minimal geodesic. Finally, we see for any $\tau\in [0, 1]$,
\begin{align*}
    \scrmk_{p, q}(\delta^\omega_{E, y_0}\otimes\sigma, \mathfrak{m}_\tau)
    &\leq \scrmk_{p, q}(\delta^\omega_{E,y_0}\otimes\sigma, \mathfrak{m}_0)+\scrmk_{p, q}(\mathfrak{m}_0, \mathfrak{m}_\tau)\\
    &\leq \scrmk_{p, q}(\delta^\omega_{E,y_0}\otimes\sigma, \mathfrak{m}_0)+\tau\scrmk_{p, q}(\mathfrak{m}_0, \mathfrak{m}_1)\\
    &<\infty,
\end{align*}
hence $\mathfrak{m}_\tau\in \mathcal{P}^\sigma_{p, q}(E)$. Thus $\tau\mapsto \mathfrak{m}_\tau$ is a minimal geodesic with respect to~$\scrmk_{p, q}$.
\end{proof}
\subsection{Duality}
We now work toward a duality result for disintegrated Monge--Kantorovich metrics.
%

We begin by showing the space $\mathcal{X}_p$ is well-defined.
\begin{lemma}\label{lem: spaces well defined}
 The space $\mathcal{X}_p$ is a Banach space, independent of the choices of $\{U_j\}_{j\in \N}$, $\{\Xi_j\}_{j\in \N}$, $\{\chi_j\}_{j\in \N}$, and $y_0\in Y$, and the associated norm $\left\| \cdot\right\|_{\mathcal{X}_p}$ will be bi-Lipschitz equivalent under a different choice of the above.
\end{lemma}
\begin{proof}
    Again let $\{\widetilde{U}_j\}_{j\in \N}$, $\{\widetilde{\Xi}_j\}_{j\in \N}$, $\{\tilde{\chi}_j\}_{j\in \N}$, $\tilde{y}_0\in Y$, $\dist_{E,\tilde{y}_0}^p$ be alternate choices of the relevant objects. 
For each $\omega \in U_j\cap U_{j'}$ with $j,j'\in\mathbb{N}$,
there exists $\gamma_j^{j'}(\omega)\in G$ such that
\begin{align*}
\widetilde{\Xi}_{j',\omega}^{-1}(\Xi_{j,\omega}(y))=\gamma_j^{j'}(\omega) y
\end{align*}
for $y\in Y$.
Then for any $u\in E$ and $\omega\in \Omega$, we have
\begin{align*}
\dist_{E,\tilde{y}_0}^p(\omega, u) 
&
=\sum_{j\in \mathbb{N}}
\chi_j(\omega)
\dist_{E,\tilde{y}_0}^p(\omega, u)\\
&
\leq
2^{p-1}\sum_{j\in \mathbb{N}} \chi_j(\omega)
\left( 
\dist_{E,\tilde{y}_0}^p(\omega, \Xi_{j,\omega}(y_0)  ) 
+
\dist_E(\Xi_{j,\omega}(y_0),u)^p\right)\\
&\leq 2^{p-1}\sum_{j, j'\in \N}\tilde{\chi}_j(\omega)\chi_{j'}(\omega)\dist_E(\widetilde{\Xi}_{j', \omega}(\tilde{y}_0), \Xi_{j, \omega}(y_0))^p
+2^{p-1}\dist_{E,y_0}^p(\omega,u)\\
&\leq 2^{p-1}\sum_{j, j'\in \N}\tilde{\chi}_j(\omega)\chi_{j'}(\omega)\dist_Y(\tilde{y_0}, \gamma_j^{j'}(\omega)y_0)^p
+2^{p-1}\dist_{E,y_0}^p(\omega,u).
\end{align*}
    The first term above is bounded independent of $u$ and $\omega$ (depending only on $y_0$ and $\tilde{y}_0$) by~\eqref{eqn: bounded orbits}, hence there is some constant $C>0$ such that 
    \begin{align*}
        1+
       \dist_{E,\tilde{y}_0}^p(\pi(u), u)
\leq C(1+\dist_{E,y_0}^p(\pi(u), u)),
    \end{align*}
for all $u\in E$, which proves the lemma.
\end{proof}
Next we define a subspace of $C(Y)$ 
assuming $(Y,\dist_Y)$ is locally compact, by 
\begin{align*}
\mathcal{Y}_p
\coloneqq \left\{ \phi \in C(Y)
\Bigm|
\frac{\phi (t) }{1+\dist_{Y}(y_0,t)^p} \in C_0(Y)
\text{ for some (hence all) $y_0\in Y$}
\right\}
\end{align*}
equipped with the norm defined by 
\[
\left\| \phi \right\|_{\mathcal{Y}_p,y_0}\coloneqq \sup_{t\in Y}\left|\frac{\phi (t) }{1+\dist_Y(y_0,t)^p} \right|
\quad \text{ for } \phi\in C(Y).
\]
Since all $(\mathcal{Y}_p, \left\| \cdot \right\|_{\mathcal{Y}_p,y_0})$ for $y_0\in Y$
are equivalent to each other, 
we simply denote this normed space by $\mathcal{Y}_p$ and write the norm as $\left\| \cdot \right\|_{\mathcal{Y}_p}$ with the convention that we have fixed some $y_0\in Y$, 
when there is no possibility of confusion. It is easy to see that $(\mathcal{Y}_p, \|\cdot\|_{\mathcal{Y}_p})$ is a Banach space. 

We now recall the \emph{Kantorovich duality} for $\mk_p^X$ on a metric space $(X, \dist_X)$, which will be the basis of a duality theory for $\scrmk_{p, q}$.
\begin{theorem}[\cite{Villani09}*{Theorem~5.10}]\label{Kantorovich}
Let $(X, \dist_X)$ be a complete, separable metric space, and $1\leq p<\infty$, then for $\mu,\nu\in \mathcal{P}(X)$,  
\begin{align*}
\mk_p^X(\mu,\nu)^p
&=\sup
\left\{
-\int_X\phi d\mu-\int_X\psi d\nu
\Biggm|
\begin{tabular}{ll}
$(\phi,\psi)\in C_b(X)^2$,\\
$-\phi(t)-\psi(s) \leq \dist_X(t, s)^p$ 
for $(t,s)\in X^2$
\end{tabular}
\right\}\\
&=\sup
\left\{
-\int_X\psi^{\dist_X^p} d\mu-\int_X\psi d\nu
\Bigm|
\phi\in C_b(X)
\right\}.
\end{align*}
\end{theorem}
Also recall the following definition.
\begin{definition}\label{def: c-transform}
    For a function $\phi$ on a metric space $(X, \dist_X)$ and $s\in X$,
 the \emph{$\dist_X^p$-transform of $\phi$} is defined by
\begin{align*}
\phi^{\dist_X^p}(\tilde{x})\coloneqq \sup_{x\in X} \left(-\dist_X(x,\tilde{x})^p-\phi(x)\right) \in (-\infty,\infty].
\end{align*}
\end{definition}

Next we show a few lemmas on the $\dist_Y^p$-transform of a function in $\mathcal{Y}_p$. The continuity below is an analogue of \cite{GangboMcCann96}*{Appendix C}, but in spaces other than $\R^n$ and for functions in the restricted class~$\mathcal{Y}_p$.
\begin{lemma}\label{lem: Xp transform bdd}
If $\phi\in \mathcal{Y}_p$, then $\phi^{\dist_Y^p}$ is locally bounded and continuous on $Y$,
and belongs to $L^1(\mu)$ for all $\mu\in \mathcal{P}_p(Y)$. 
\end{lemma}
\begin{proof}
We first show local boundedness. Note by definition, 
\[
\phi^{\dist_Y^p}(s)\geq -\dist_Y(s,s)^p-\phi(s)=-\phi(s)>-\infty
\]
for all $s\in Y$.
To see local boundedness from above, fix $y_0,s\in Y$. 
Since compact sets are bounded and $\phi\in \mathcal{Y}_p$, there exists an $R>0$ such that  
if $\dist_{y_0}(t)>R$, then 
\[
\frac{\left|\phi(t)\right|}{1+\dist_{y_0}(t)^p}\leq 2^{-p},
\]
we calculate for such $t$,
\begin{align}\label{eqn: tail bound}
\begin{split}
 -\dist_Y(t, s)^p-\phi(t)
 &
 \leq
 -\dist_Y(t, s)^p+2^{-p}\left(1+\dist_{y_0}(t)^p\right)\\
&\leq
 -\dist_Y(t, s)^p+
 2^{-p}\left[1+2^{p-1}\left(\dist_Y(t,s)^p+\dist_{y_0}(s)^p\right)\right]\\
 &=-\frac12 \dist_Y(t, s)^p+\frac{1}{2^p}+\frac12\dist_{y_0}(s)^p\\
&\leq \frac{1}{2^p}+\frac{1}{2}\dist_{y_0}(s)^p.
\end{split}
\end{align}
Thus
\begin{align*}
 \phi^{\dist_Y^p}(s)\leq \max\left\{
 \frac{1}{2^{p}}+\frac{1}{2}\dist_{y_0}(s)^p, \sup_{t\in B_R^Y(y_0)}\left(-\dist_Y(t, s)^p-\phi(t)\right)\right\},
\end{align*}
since $\phi \in \mathcal{Y}_p$ implies $\phi$ is bounded on bounded, open balls, the expression on the right is locally bounded in $s$, hence we see $\phi^{\dist_Y^p}$ is locally bounded. Since $\mu$ has finite $p$th moment, the above bounds give $\phi^{\dist_Y^p}\in L^1(\mu)$.

To see continuity, fix a convergent sequence $(s_\ell)_{\ell\in\mathbb{N}}$ in $Y$ with limit $s_0$ and fix $\varepsilon>0$.
Then since $\phi^{\dist_Y^p}$ is locally bounded from above, there exists $t_0\in Y$ such that 
\begin{align*}
\phi^{\dist_Y^p}(s_0)\leq -\dist_Y(t_0, s_0)^p-\phi(t_0)+\varepsilon,
\end{align*}
thus
\begin{align*}
    \phi^{\dist_Y^p}(s_0)-\phi^{\dist_Y^p}(s_\ell)
    &\leq -\dist_Y(t_0, s_0)^p+\dist_Y(t_0, s_\ell)^p+\varepsilon\\
    &\leq p\cdot\max\{\dist_Y(t_0, s_\ell)^{p-1}, \dist_Y(t_0, s_0)^{p-1}\}\left|\dist_Y(t_0, s_\ell)-\dist_Y(t_0, s_0)\right|+\varepsilon\\
    &\leq p\cdot\max\{\dist_Y(t_0, s_\ell)^{p-1}, \dist_Y(t_0, s_0)^{p-1}\}\dist_Y(s_\ell, s_0)+\varepsilon\\
    &< 2\varepsilon
\end{align*}
if $\ell$ is sufficiently large. Similarly, for any $\ell\in \N$,
we have 
\begin{align*}
    \phi^{\dist_Y^p}(s_\ell)-\phi^{\dist_Y^p}(s_0)
    \leq p\max\{\dist_Y(t_\ell, s_\ell)^{p-1}, \dist_Y(t_\ell, s_0)^{p-1}\}\dist_Y(s_\ell, s_0)+\varepsilon,
\end{align*}
where $t_\ell\in Y$ satisfies
\[
\phi^{\dist_Y^p}(s_\ell)\leq -\dist_Y(t_\ell, s_\ell)^p-\phi(t_\ell)+\varepsilon.
\]
Now suppose by contradiction that (after passing to some subsequence) $\lim_{\ell\to\infty}\dist_{y_0}(t_{\ell})=\infty$, then since $\phi\in \mathcal{Y}_p$, for all~$\ell$ sufficiently large we can apply~\eqref{eqn: tail bound} to obtain
\begin{align*} 
-\phi^{\dist_Y^p}(s_\ell)
\leq
-\dist_Y(t_{\ell}, s_\ell)^p-\phi(t_{\ell})+\varepsilon
\leq 
-\frac{1}{2}\dist_Y(t_{\ell}, s_\ell)^p+\frac{1}{2^p}+\frac12\dist_{y_0}(s_\ell)^p+\varepsilon
\xrightarrow{\ell\to \infty}-\infty,
\end{align*}
as $(s_\ell)_{\ell\in \N}$ is bounded. 
This contradicts that $\phi^{\dist_Y^p}$ is locally bounded, since $s_\ell\to s_0$ as $\ell\to\infty$. Thus for $\ell$ sufficiently large, 
\begin{align*}
    \phi^{\dist_Y^p}(s_\ell)-\phi^{\dist_Y^p}(s_0)
    &\leq p\max\{\dist_Y(t_\ell, s_\ell)^{p-1}, \dist_Y(t_\ell, s_0)^{p-1}\}\dist_Y(s_\ell, s_0)+\varepsilon<2\varepsilon,
\end{align*}
and we see $\phi^{\dist_Y^p}$ is continuous at $s_0$.
\end{proof} 
Next we prove stability of $\dist_Y^p$-transforms under the norm of $\mathcal{Y}_p$. Note we do not claim that $\tilde{\phi}^{\dist_Y^p}$ belongs to $\mathcal{Y}_p$ in~\eqref{eqn: Xp transform conv} below.
\begin{lemma}\label{lem: c_p transform of X_p}
Let $\phi\in \mathcal{Y}_p$ and $\mu \in \mathcal{P}_p(Y)$. Then: 
\begin{enumerate}
\setlength{\leftskip}{-15pt}
\item\label{eqn: Xp transform L1}
$\phi \in L^1(\mu)$ and 
\[
 \int_Y \left|\phi\right| d\mu
 \leq 
 \left\|\phi\right\|_{\mathcal{Y}_p}\int_Y(1+\dist_{y_0}(t)^p)d\mu(t).
 \]
\item\label{eqn: Xp transform conv}
Let $R_\phi>0$ be such that 
if $\dist_{y_0}(t)>R_\phi$, then 
\[
\frac{|\phi(t)|}{1+\dist_{y_0}(t)^p}
\leq 2^{-p-1}.
\]
Then for all
$\tilde{\phi}\in \mathcal{Y}_p$ with
$\|\phi-\tilde{\phi}\|_{\mathcal{Y}_p}<2^{-p-1}$ and $s\in Y$,
\begin{align*}
|\tilde{\phi}^{\dist_Y^p}(s)-\phi^{\dist_Y^p}(s)|
&\leq\left\| \phi-\tilde{\phi}\right\|_{\mathcal{Y}_p}
\left(1+
\max\{
R_\phi^p,
2^{p+1}(1+\left\| \phi\right\|_{\mathcal{Y}_p})(1+\dist_{y_0}(s)^p)\}\right).
\end{align*}
\end{enumerate}
\end{lemma}
\begin{proof}
Assertion~\eqref{eqn: Xp transform L1} follows from the inequality 
\[
\left|\phi(t)\right| \leq \|\phi\|_{\mathcal{Y}_p} (1+\dist_{y_0}(t)^p)
\quad
\text{for all }t\in Y.
\]

Assertion~\eqref{eqn: Xp transform conv} is more involved. 
Fix $\varepsilon>0$, then if $\tilde{\phi}\in \mathcal{Y}_p$ 
by Lemma~\ref{lem: Xp transform bdd},   $\tilde{\phi}^{\dist_Y^p}$ is finite on all of $Y$. Thus for any $s\in Y$, there exists $t_{\tilde{\phi}}\in Y$ such that
\begin{align*}
    \tilde{\phi}^{\dist_Y^p}(s)&\leq -\dist_Y(t_{\tilde{\phi}},s)^p-\phi(t_{\tilde{\phi}})+\varepsilon.
\end{align*}
Then, 
\begin{align*}
   \tilde{\phi}^{\dist_Y^p}(s)-\phi^{\dist_Y^p}(s)
   &\leq -\dist_Y(t_{\tilde{\phi}},s)^p-\tilde{\phi}(t_{\tilde{\phi}})
 +\dist_Y( t_{\tilde{\phi}},s)^p
 +\phi(t_{\tilde{\phi}})+\varepsilon
\leq\left\| \phi-\tilde{\phi}\right\|_{\mathcal{Y}_p}(1+\dist_{y_0}(t_{\tilde{\phi}})^p)+\varepsilon,
\end{align*}
and switching the roles of $\phi$, $\tilde{\phi}$ yields
\begin{align}
\label{eqn: transform integral bound}
|\tilde{\phi}^{\dist_Y^p}(s)-\phi^{\dist_Y^p}(s)|
\leq\left\| \phi-\tilde{\phi}\right\|_{\mathcal{Y}_p}\left(1+
\max\{
\dist_{y_0}(t_{\phi})^p,\dist_{y_0}(t_{\tilde{\phi}})^p
\}\right)
+\varepsilon.
\end{align}
Now suppose $\tilde{\phi}\in \mathcal{Y}_p$ with $\|\phi-\tilde{\phi}\|_{\mathcal{Y}_p}<2^{-p-1}$,
then if $\dist_{y_0}(t)>R_\phi$,
\begin{align*}
\frac{\left|\tilde{\phi}(t)\right|}{1+\dist_{y_0}(t)^p}
\leq 
\|\phi-\tilde{\phi}\|_{\mathcal{Y}_p}
+
\frac{\left|\phi(t)\right|}{1+\dist_{y_0}(t)^p}
<2^{-p}.
\end{align*}
If  $s$, $t\in Y$ satisfy $\dist_{y_0}(t) \geq \max\{R_\phi, 2\dist_{y_0}(s)\}$,
 by the triangle inequality, 
\begin{align*}
    \dist_Y(t, s)\geq \left| \dist_{y_0}(t)-\dist_{y_0}(s)\right| =\dist_{y_0}(t)-\dist_{y_0}(s)\geq \frac{1}{2}\dist_{y_0}(t),
\end{align*}
then from \eqref{eqn: tail bound} we obtain that
\begin{align*}
-\dist_Y(t,s)^p-\tilde{\phi}(t)
\leq
-\frac12 \dist_Y(t, s)^p+\frac{1}{2^p}+\frac{1}{2}\dist_{y_0}(s)^p
\leq -\frac{1}{2^{p+1}}\dist_{y_0}(t)^p+ \frac{1}{2^{p}}+\frac{1}{2}\dist_{y_0}(s)^p,
\end{align*}
Thus if $s\in Y$ is such that
$\dist_{y_0}(t_{\tilde{\phi}})\geq \max\{R_\phi,2\dist_{y_0}(s)\}$, 
we have 
\begin{align*}
 -\left\| \tilde{\phi}\right\|_{\mathcal{Y}_p}(1+\dist_{y_0}(s)^p)
 \leq-\tilde{\phi}(s)
 \leq \tilde{\phi}^{\dist_Y^p}(s)
 &\leq -\dist_Y(t_{\tilde{\phi}},s)^p-\tilde{\phi}(t_{\tilde{\phi}})+\varepsilon\\
&\leq -\frac{1}{2^{p+1}}\dist_{y_0}(t_{\tilde{\phi}})^p+ \frac{1}{2^{p}}+\frac{1}{2}\dist_{y_0}(s)^p+\varepsilon
\end{align*}
or rearranging,
\begin{align*}
\dist_{y_0}(t_{\tilde{\phi}})^p
&\leq 2^{p+1}\left\| \tilde{\phi}\right\|_{\mathcal{Y}_p}(1+\dist_{y_0}(s)^p) 
+ 2+2^{p}\dist_{y_0}(s)^p+2^{p+1}\varepsilon\\
&\leq 
2^{p+1}(2^{-p-1}+\left\| {\phi}\right\|_{\mathcal{Y}_p})(1+\dist_{y_0}(s)^p) 
+ 2+2^{p}\dist_{y_0}(s)^p+2^{p+1}\varepsilon\\
&
\leq 2^{p+1}
\left[(1+\left\| \phi\right\|_{\mathcal{Y}_p})(1+\dist_{y_0}(s)^p)+\varepsilon\right].
\end{align*}
Thus in all cases, we have
\begin{align*}
    \dist_{y_0}(t_{\tilde{\phi}})^p
&\leq
\max\left\{
R_\phi^p,
2^{p+1}\left[(1+\left\| \phi\right\|_{\mathcal{Y}_p})(1+\dist_{y_0}(s)^p)+\varepsilon\right]\right\}.
\end{align*}
We can obtain the above estimate when $\tilde{\phi}=\phi$ as well, hence combining with \eqref{eqn: transform integral bound} and taking $\varepsilon$ to~$0$ finishes the proof.
\end{proof}

Our approach will be to apply the classic Kantorovich duality for each $\omega\in \Omega$, and appeal to the Kuratowski and Ryll-Nardzewski measurable selection theorem (Theorem~\ref{measurableselection}) to obtain the necessary measurability. However, care must be taken to utilize this measurability since we are not in the trivial bundle case. To this end, given $\mathfrak{m}$, $\mathfrak{n}\in \mathcal{P}^\sigma_p(E)$, and $\varepsilon>0$, 
for each $j\in \N$ we define a set-valued function $\overline{F}^{\mathfrak{m}, \mathfrak{n}}_{j, \varepsilon}$ from $U_j$ to~$2^{\mathcal{Y}_p}$ by
\begin{align*}
\overline{F}^{\mathfrak{m}, \mathfrak{n}}_{j, \varepsilon}(\omega)
&\coloneqq \overline{\left\{
\phi \in \mathcal{Y}_p\Biggm|
 -\int_Y\phi d(\Xi^{-1}_{j, \omega})_\sharp\mathfrak{m}^\omega-\int_Y\phi^{\dist_Y^p} d(\Xi^{-1}_{j, \omega})_\sharp\mathfrak{n}^\omega
 > \mk_p^E(\mathfrak{m}^\omega, \mathfrak{n}
 ^\omega)^p-\varepsilon\right\}}^{\left\|\cdot\right\|_{\mathcal{Y}_p}},
\end{align*}
where $\overline{A}^{\left\|\cdot\right\|_{\mathcal{Y}_p}}$ denotes the closure of $A\subset \mathcal{Y}_p$ 
with respect to the norm $\left\|\cdot\right\|_{\mathcal{Y}_p}$.

For the remainder of the section, for $j\in \N$ we denote
\begin{align*}
    \sigma_j\coloneqq \sigma\vert_{U_j}.
\end{align*}
\begin{lemma}\label{lem: weakly measurable}
Assume $(Y, \dist_Y)$ is locally compact and let $\mathfrak{m}$, $\mathfrak{n}\in \mathcal{P}^\sigma_p(E)$. 
Then for $\varepsilon>0$ and $j\in\N$, it holds that $\overline{F}^{\mathfrak{m}, \mathfrak{n}}_{j, \varepsilon}$ is $\mathcal{B}_{\sigma_j}$-weakly measurable
and $\overline{F}^{\mathfrak{m}, \mathfrak{n}}_{j, \varepsilon} (\omega)$ is closed and nonempty for $\sigma$-a.e. $\omega \in U_j$.
\end{lemma}
\begin{proof}
Since 
$\mathfrak{m}$, $\mathfrak{n}\in \mathcal{P}^\sigma_p(E)$, $j\in\N$, and $\varepsilon>0$ are fixed, 
we write $\overline{F}$ in place of $\overline{F}^{\mathfrak{m}, \mathfrak{n}}_{j, \varepsilon}$.
We first show 
$\overline{F}(\omega)\neq \emptyset$ for $\sigma_j$-a.e. $\omega\in U_j$.
Since $(\Xi^{-1}_{j, \omega})_\sharp\mathfrak{m}^\omega$, $(\Xi^{-1}_{j, \omega})_\sharp\mathfrak{n}^\omega\in \mathcal{P}_p(Y)$ for $\sigma_j$-a.e.~$\omega$, for such $\omega$ we have
\begin{align*}
\mk_p^E(\mathfrak{m}^\omega, \mathfrak{n}^\omega)=\mk_p^Y((\Xi^{-1}_{j, \omega})_\sharp\mathfrak{m}^\omega, (\Xi^{-1}_{j, \omega})_\sharp\mathfrak{n}^\omega)<\infty
\end{align*}
and by the Kantorovich duality Theorem~\ref{Kantorovich} for $\mk_p^Y$,  
there exists
$\phi_\varepsilon\in C_b(Y)\subset \mathcal{Y}_p$ 
such that 
\begin{align*}
\mk_p^E(\mathfrak{m}^\omega, \mathfrak{n}^\omega)^p-
\varepsilon
&
<
-\int_Y\phi^{\dist_Y^p}_\varepsilon d(\Xi^{-1}_{j, \omega})_\sharp\mathfrak{m}^\omega-\int_Y\phi_\varepsilon d(\Xi^{-1}_{j, \omega})_\sharp\mathfrak{n}^\omega,
\end{align*}
thus $\phi_\varepsilon \in \overline{F}(\omega)\neq \emptyset$.
By definition, $\overline{F}(\omega)$ is closed.

Next, we prove the $\mathcal{B}_{\sigma_j}$-weak measurability of $\overline{F}$. 
Define 
\begin{align*}
F(\omega)
&\coloneqq \left\{
\phi \in \mathcal{Y}_p\Biggm|
 -\int_Y\phi^{\dist_Y^p} d(\Xi^{-1}_{j, \omega})_\sharp\mathfrak{m}^\omega-\int_Y\phi d(\Xi^{-1}_{j, \omega})_\sharp\mathfrak{n}^\omega
 > \mk_p^E(\mathfrak{m}^\omega, \mathfrak{n}
 ^\omega)^p-\varepsilon
\right\}.
\end{align*}
First, for any open set $O\subset \mathcal{Y}_p$ and any set $A\subset \mathcal{Y}_p$, 
it trivially holds  that $\overline{A}^{\left\|\cdot\right\|_{\mathcal{Y}_p}}\cap O\neq \emptyset$ if and only if $A\cap O\neq\emptyset$, 
thus it is sufficient to prove that $F$ is $\mathcal{B}_{\sigma_j}$-weakly measurable.
To this end, fix $\phi\in\mathcal{Y}_p$ and define the function $G_\phi: \Omega\to [-\infty,\infty)$ by
\begin{align*}
 G_\phi(\omega)\coloneqq -\int_Y\phi^{\dist_Y^p} d(\Xi^{-1}_{j, \omega})_\sharp\mathfrak{m}^\omega-\int_Y\phi d(\Xi^{-1}_{j, \omega})_\sharp\mathfrak{n}^\omega-\mk_p^E(\mathfrak{m}^\omega, \mathfrak{n}^\omega)^p,
\end{align*}
then $\phi\in F(\omega)$ if and only if $G_\phi(\omega)>-\varepsilon$, hence
\begin{align}\label{eqn: F inverse decomposition}
 \{ \omega\in \Omega\mid  F(\omega) \cap O \neq \emptyset \}&=\bigcup_{\phi\in O}G_\phi^{-1}((-\varepsilon, \infty)).
\end{align}
Since $(Y, \dist_Y)$ is locally compact and separable, by combining \cite{Kechris95}*{(5.3) Theorem ii) and iv)}, and ~\cite{Conway90}*{Chapter V.5, Exercise 2(c)} we find $C_0(Y)$ is separable, hence there exists a countable set $\{\tilde{\phi}_{\ell}\}_{\ell\in\mathbb{N}}\subset C_0(Y)$, dense in 
the supremum norm, then 
\begin{align*}
\{\phi_{\ell}\}_{\ell\in\mathbb{N}}\coloneqq \left\{ (1+\dist_{y_0}^p)\tilde{\phi}_{\ell}\right\}_{\ell\in\mathbb{N}}\subset \mathcal{Y}_p
\end{align*}
is dense in $\left\|\cdot\right\|_{\mathcal{Y}_p}$; we may throw out some elements to assume $\{\phi_\ell\}_{\ell\in\mathbb{N}}\subset O$ while remaining dense in $O$. 
 We now claim that
\begin{align}\label{eqn: measurability inclusions}
\bigcup_{\phi\in O}G_\phi^{-1}((-\varepsilon, \infty))
 =\bigcup_{\ell=1}^\infty G_{\phi_\ell}^{-1}((-\varepsilon, \infty)).
\end{align}
Since $\{\phi_\ell\}_{\ell\in\mathbb{N}}\subset O$, it is clear that
\[
\bigcup_{\ell=1}^\infty G_{\phi_\ell}^{-1}((-\varepsilon, \infty))
\subset
\bigcup_{\phi\in O}G_\phi^{-1}((-\varepsilon, \infty)). 
\]
On the other hand, suppose $\omega\in G_\phi^{-1}((-\varepsilon, \infty))$ for some $\phi\in O$. From Lemma~\ref{lem: c_p transform of X_p}
combined with the fact that $(\Xi^{-1}_{j, \omega})_\sharp\mathfrak{n}^\omega\in\mathcal{P}_p(Y)$, and the density of $\{\phi_\ell\}_{\ell\in \mathbb{N}}$ in $\mathcal{Y}_p$, 
for any $\delta>0$, there exists $\ell_\delta\in \N$ such that
\[
G_{\phi}(\omega)-G_{\phi_{\ell_\delta}}(\omega)
=
-\int_Y \left(\phi^{\dist_Y^p}-\phi^{\dist_Y^p}_{\ell_\delta}\right) d(\Xi^{-1}_{j, \omega})_\sharp\mathfrak{m}^\omega-
\int_Y \left(\phi-\phi_{\ell_\delta}\right) d(\Xi^{-1}_{j, \omega})_\sharp\mathfrak{n}^\omega
<\delta,
\]
thus taking $\delta=G_{\phi}(\omega)+\varepsilon>0$, we have 
\[
G_{\phi}(\omega)-G_{\phi_{\ell_\delta}}(\omega)
<G_{\phi}(\omega)+\varepsilon,
\]
consequently  $G_{\phi_{\ell_\delta}}(\omega)>-\varepsilon$. 
Thus $\omega\in  G_{\phi_{\ell_\delta}}^{-1}((-\varepsilon, \infty))$ and the opposite inclusion is proved.

By \cite{AmbrosioGigliSavare08}*{Lemma 12.4.7} and \nameref{thm: disintegration}, we see that 
$
G_{\phi_\ell}^{-1}((-\varepsilon, \infty))\in\mathcal{B}_{\sigma_j}$
for each $\ell\in \mathbb{N}$, 
hence 
\[
\bigcup_{\ell=1}^\infty G_{\phi_\ell}^{-1}((-\varepsilon, \infty))\in \mathcal{B}_{\sigma_j}.
\]
Thus combining \eqref{eqn: F inverse decomposition} and \eqref{eqn: measurability inclusions}, this shows $F$ is $\mathcal{B}_{\sigma_j}$-weakly measurable.
\end{proof}

We now prove some auxiliary lemmas.
\begin{lemma}\label{lem: measurability of transform}
For $j\in \N$, if $f\in L^0(\sigma_j; \mathcal{Y}_p)$, then for $\mathfrak{m}$, $\mathfrak{n}\in \mathcal{P}^\sigma_p(E)$, the functions defined by
\begin{equation}\label{functions}
\omega\mapsto \int_Y f_\omega^{\dist_Y^p} d(\Xi^{-1}_{j, \omega})_\sharp\mathfrak{m}^\omega,
\qquad
\omega\mapsto \int_Y f_\omega d(\Xi^{-1}_{j, \omega})_\sharp\mathfrak{n}^\omega
\end{equation}
are $\mathcal{B}_{\sigma_j}$-measurable on $U_j$.
\end{lemma}
\begin{proof}
Since $f$ is $\sigma_j$-strongly measurable, for each $\ell\in \N$ there exist $I_\ell\in \N$, $\{\phi_{i, \ell}\}_{i=1}^{I_\ell}\subset \mathcal{Y}_p$, and a partition $\{A_{i, \ell}\}_{i=1}^{I_\ell}\subset \mathcal{B}_{\sigma_j}$ of $U_j$ so that for $\sigma_j$-a.e. $\omega$, the sequence
\[
f^\ell_\omega\coloneqq \sum_{i=1}^{I_\ell}\mathds{1}_{A_{i, \ell}}(\omega)\phi_{i, \ell}
\]
converges   
to $f_\omega$ in $\left\|\cdot\right\|_{\mathcal{Y}_p}$. 
The probability measures $(\Xi^{-1}_{j, \omega})_\sharp\mathfrak{m}^{\bullet}$ and $(\Xi^{-1}_{j, \omega})_\sharp\mathfrak{n}^{\bullet}$ have finite $p$th moment $\sigma$-a.e.,
fix $\omega$ such that this holds. For each $\ell\in \mathbb{N}$,
since $\{A_{i, \ell}\}_{i=1}^{I_\ell}$ is a disjoint collection
there exists a unique $1\leq i_\ell\leq I_\ell$ such that $\omega\in A_{i_\ell, \ell}$, 
then 
\begin{align*}
\int_Y f^\ell_\omega d(\Xi^{-1}_{j, \omega})_\sharp\mathfrak{n}^\omega&=\sum_{i=1}^{I_\ell}\mathds{1}_{A_{i, \ell}}(\omega)\int_Y\phi_{i, \ell}(t)d(\Xi^{-1}_{j, \omega})_\sharp\mathfrak{n}^\omega(t)
\end{align*}
and 
\begin{align*}
\int_Y(f^\ell_{\omega})^{\dist_Y^p}d(\Xi^{-1}_{j, \omega})_\sharp\mathfrak{m}^\omega
 &=\int_Y \left[
      \sup_{t\in Y}
      \left(-\dist_Y(t, s)^p-\sum_{i=1}^{I_\ell}\mathds{1}_{A_{i, \ell}}({\omega})\phi_{i, \ell}(t)\right)\right]
      d(\Xi^{-1}_{j, \omega})_\sharp\mathfrak{m}^\omega(s)\\
 &=\int_Y
 \left[
  \sup_{t\in Y}\left(-\dist_Y(t, s)^p-\phi_{i_\ell, \ell}(t)\right)\right]d(\Xi^{-1}_{j, \omega})_\sharp\mathfrak{m}^\omega(s)\\
 &=\int_Y\phi_{i_\ell, \ell}^{\dist_Y^p}d(\Xi^{-1}_{j, \omega})_\sharp\mathfrak{m}^\omega
 =\sum_{i=1}^{I_\ell}\mathds{1}_{A_{i, \ell}}({\omega})\int_Y\phi_{i, \ell}^{\dist_Y^p}d(\Xi^{-1}_{j, \omega})_\sharp\mathfrak{m}^\omega,
\end{align*}
which are $\mathcal{B}_{\sigma_j}$-measurable functions of $\omega\in U_j$ by \nameref{thm: disintegration}. 
Thus 
from Lemma~\ref{lem: c_p transform of X_p}, we observe each of the functions in \eqref{functions} 
is a $\sigma$-a.e. pointwise limit of $\mathcal{B}_{\sigma_j}$-measurable functions, hence is $\mathcal{B}_{\sigma_j}$-measurable itself.
\end{proof}
\begin{lemma}\label{lem: continuous approximation of strongly measurable}
If $f\in L^0(\sigma_j; \mathcal{Y}_p)$, 
there is a sequence $(f_\ell)_{\ell\in \mathbb{N}}\subset C_b(U_j; \mathcal{Y}_p)$ which converges pointwise $\sigma_j$-a.e. to $f$.
\end{lemma}
\begin{proof}
By Remark~\ref{V58}, $f$ is a $\mathcal{B}_{\sigma_j}$-measurable map. Then since $\mathcal{Y}_p$ is complete and separable, for each $\ell\in \N$, we may apply~\cite{Bogachev07}*{Theorem 7.1.13},
where $\mathcal{B}_\mu(X)$ in the reference is our $\mathcal{B}_{\sigma_j}$, 
to~$f$ to find a compact set $K_\ell\subset U_j$ such that $\sigma_j(U_j\setminus K_\ell)<2^{-\ell}$ and $f$ restricted to $K_\ell$ is continuous; we may also assume $K_\ell\subset K_{\ell+1}$ for each $\ell\in \mathbb{N}$. Since $\mathcal{Y}_p$ is a normed space it is locally convex, hence the Tietze extension theorem \cite{Dugundji51}*{Theorem 4.1} applies and there is a continuous function $f_\ell: U_j\to \mathcal{Y}_p$ such that $f_\ell= f$ on $K_\ell$. Moreover since $K_\ell$ is compact and $f$ restricted to it is continuous, the image $f(K_\ell)$ is also compact, hence bounded in $\mathcal{Y}_p$. Then~\cite{Dugundji51}*{Theorem 4.1} also ensures that the image $f_\ell(U_j)$ is contained in the convex hull of~$f(K_\ell)$, consequently $f_\ell$ is bounded. Since $\sigma_j(K_\ell)\to \sigma_j(U_j)$ as $\ell\to \infty$, 
it is clear that $f_\ell$ converges pointwise $\sigma_j$-a.e. to $f$,  
finishing the proof.
\end{proof}
We are now ready to prove the duality result. Note carefully that we do not require $\mathfrak{m}$ and $\mathfrak{n}$ to belong to $\mathcal{P}^\sigma_{p, q}(E)$, but only to $\mathcal{P}^\sigma_p(E)$. This will be relevant for Corollary~\ref{cor: disint lsc} below.
\begin{proof}[Proof of Theorem~\ref{thm: main disint}~\eqref{thm: disint duality}]
Recall $r=p/q$, $\mathfrak{m}$, $\mathfrak{n}\in \mathcal{P}^\sigma_p(E)$, and we first assume $(Y, \dist_Y)$ is locally compact.
Let $(\Phi, \Psi)\in \mathcal{A}_{p,E,\sigma}$. Since $\mathfrak{m}^\omega$, $\mathfrak{n}^\omega\in \mathcal{P}_p(\pi^{-1}(\{\omega\}))$ for $\sigma$-a.e. $\omega$, by the Kantorovich duality Theorem~\ref{Kantorovich} for $\mk_p^E$ restricted to $\pi^{-1}(\{\omega\})$, and the dual representation for the $L^r$ norm again (\cite{Folland99}*{Proposition 6.13}) we have 
\begin{align*}
 -\int_\Omega\zeta(\omega)\left(\int_E\Phi(u)d\mathfrak{m}^\omega(u)
+\int_E\Psi(v)d\mathfrak{n}^\omega(v)\right)d\sigma(\omega)
&\leq \int_\Omega\zeta(\omega)\mk_p^Y(\mathfrak{m}^\omega, \mathfrak{n}^\omega)^pd\sigma(\omega)\\
&\leq \left\| \mk_p^E(\mathfrak{m}^\bullet, \mathfrak{n}^\bullet)^p\right\|_{L^{r}(\sigma)}\\
&=\left\| \mk_p^E(\mathfrak{m}^\bullet, \mathfrak{n}^\bullet)\right\|_{L^q(\sigma)}^p=\scrmk_{p, q}(\mathfrak{m}, \mathfrak{n})^p.
\end{align*}

To show the reverse inequality, fix $\varepsilon>0$ and let $\Omega'$ be the set of $\omega\in \Omega$ such that both of~$\mathfrak{m}^\omega, \mathfrak{n}^\omega$ have finite $p$th moment. 
By Lemma~\ref{lem: weakly measurable}, 
for each $j\in \N$ the set-valued mapping~$\overline{F}^{\mathfrak{m}, \mathfrak{n}}_{j, \varepsilon}$ on $U_j$ is nonempty and closed valued $\sigma$-a.e., and $\mathcal{B}_{\sigma_j}$-weakly measurable. Since  $\mathcal{Y}_p$ is separable, by Theorem~\ref{measurableselection} we can find maps $f^j_{\bullet}: U_j\to \mathcal{Y}_p$ that are $\mathcal{B}_{\sigma_j}$-measurable such that $f^j_\omega\in \overline{F}^{\mathfrak{m}, \mathfrak{n}}_{j, \varepsilon}(\omega)$ for $\sigma$-a.e. $\omega\in U_j$,
and by Remark~\ref{V58}, 
this implies $f^j_\bullet\in L^0(\sigma_j; \mathcal{Y}_p)$. 
By Lemma~\ref{lem: c_p transform of X_p} for $\omega\in \Omega'\cap U_j$
\begin{align*}
 -\int_Y (f^j_\omega)^{\dist_Y^p}(t)d(\Xi^{-1}_{j, \omega})_\sharp\mathfrak{m}^\omega(t)-\int_Yf^j_\omega(s)d(\Xi^{-1}_{j, \omega})_\sharp\mathfrak{n}^\omega(s)
 &\geq \mk_p^E(\mathfrak{m}^\omega, \mathfrak{n}^\omega)^p-\varepsilon.
\end{align*}

If $\scrmk_{p, q}(\mathfrak{m}, \mathfrak{n})<\infty$, it is easy to see there exists $\zeta\in \mathcal{Z}_{r', \sigma}$ satisfying
\begin{align*}
 \int_\Omega\zeta(\omega)\mk_p^E(\mathfrak{m}^\omega, \mathfrak{n}^\omega)^pd\sigma(\omega)&> \scrmk_{p, q}(\mathfrak{m}, \mathfrak{n})^p-\varepsilon,
\end{align*}
thus combining with the inequality above and using the properties of a partition of unity we obtain
\begin{align}
\label{eqn: disint dual inequality}
 \sum_{j\in \N}\int_\Omega\chi_j\zeta\cdot\left(-\int_Y (f^j_\bullet)^{\dist_Y^p}(t)d(\Xi^{-1}_{j, \bullet})_\sharp\mathfrak{m}^\bullet(t)-\int_Yf^j_\bullet(s)d(\Xi^{-1}_{j, \bullet})_\sharp\mathfrak{n}^\bullet(s)\right)d\sigma
 &> \scrmk_{p, q}(\mathfrak{m}, \mathfrak{n})^p-2\varepsilon
\end{align}
in the case $p=q$ we may take $\zeta\equiv 1$. 

Now for $\ell\in \N$ and $z \in \R$, let 
\begin{align*}
T_\ell(z)\coloneqq \max\{ \min\{z, \ell\}, -\ell\}=\begin{cases}\min\{z,\ell\}, &\text{if }z\geq 0,\\ \max\{ z,-\ell \}, &\text{if }z<0.\end{cases}
\end{align*}
By a simple calculation,
we see that 
for each $z_1,$ $z_2\in \mathbb{R}$,
the sequence $( T_\ell(z_1)+T_\ell(z_2) )_{\ell\in \mathbb{N}}$ is non-negative and non-decreasing if $z_1+z_2\geq0$,
and non-positive and non-increasing if $z_1+z_2 \leq 0$ with limit $z_1+z_2$, and in particular 
\begin{align}\label{minus}
\left( T_\ell(-(f^j_\omega)^{\dist_Y^p}(t))+T_\ell(-f^j_\omega(s))\right)\leq \dist_Y(t, s)^p
\end{align}
for each $t$,  $s\in Y$, $j\in \N$, and $\omega\in U_j$. 
For each $\omega\in U_j$ define the sets
 \begin{align*}
     E^j_+(\omega)\coloneqq \left\{(t,s)\bigm| f^j_\omega(t)+ (f^j_\omega)^{\dist_Y^p}(s)\leq 0\right\},
     \qquad
          E^j_-(\omega)\coloneqq \left\{(t,s)\bigm| -f^j_\omega(t)- (f^j_\omega)^{\dist_Y^p}(s)\leq 0\right\}
 \end{align*}
then we can see
\begin{align*}
\left(\pm\sum_{j\in \N}\chi_j(\omega)\int_{E^j_\pm(\omega)}\left( T_\ell(-(f^j_\omega)^{\dist_Y^p}(t))+ T_\ell(-f^j_\omega(s))\right)d((\Xi^{-1}_{j, \omega})_\sharp\mathfrak{m}^\omega\otimes(\Xi^{-1}_{j, \omega})_\sharp\mathfrak{n}^\omega)(t, s)\right)_{\ell\in \N}
\end{align*}
are non-negative, non-decreasing sequences for each $\omega\in \Omega'$.
Thus integrating against $\zeta\sigma$ and using monotone convergence 
and using the fact that 
\begin{align*}
T_\ell(-(f^j_\omega)^{\dist_Y^p}(t))+T_\ell(-f^j_\omega(s))=0
\quad
\text{on }
E^j_+(\omega)\cap E^j_-(\omega),
\end{align*}
by~\eqref{eqn: disint dual inequality} if $\ell_0$ is large enough we obtain 
\begin{align}
\begin{split}
&-\sum_{j\in \N}\int_\Omega\chi_j\zeta\cdot\left(\int_Y [-T_{\ell_0}(-f^j_\bullet)]^{\dist_Y^p}(t) d(\Xi^{-1}_{j,\bullet})_\sharp\mathfrak{m}^\bullet(t)
     +\int_Y [-T_{\ell_0}( -f^j_\bullet(s))]d(\Xi^{-1}_{j, \bullet})_\sharp\mathfrak{n}^\bullet(s) \right) d\sigma\\
    &\geq \sum_{j\in \N}\int_\Omega\chi_j\zeta\cdot\left(\int_Y T_{\ell_0}(-(f^j_\bullet)^{\dist_Y^p}(t))d(\Xi^{-1}_{j,\bullet})_\sharp\mathfrak{m}^\bullet(t)
     +\int_Y T_{\ell_0}( -f^j_\bullet(s))d(\Xi^{-1}_{j, \bullet})_\sharp\mathfrak{n}^\bullet(s)\right)d\sigma\\
    &>\scrmk_{p, q}(\mathfrak{m}, \mathfrak{n})^p-2\varepsilon,\label{eqn: Lq/p norm dual estimate 2}
\end{split}
\end{align}
where the inequality in the second line follows from \eqref{minus}, and 
the integration against $\sigma$ is justified by the measurability from by Lemma~\ref{lem: measurability of transform}. 
Let us fix such a $\ell_0$.

By Lemma~\ref{lem: continuous approximation of strongly measurable}, for each $j\in \N$ there exists a sequence $(\Psi_{j, m})_{m\in \N}$ in $C_b(U_j; \mathcal{Y}_p)$ converging pointwise $\sigma_j$-a.e. to $-T_{\ell_0}\circ (-f^j_\bullet)$ in $\left\|\cdot\right\|_{\mathcal{Y}_p}$; we may truncate to assume $\left\|(\Psi_{j, m})_\omega \right\|_{C_b(Y)} \leq 2\ell_0$,
 for all $\omega\in U_j$, 
and by~\cite{KitagawaTakatsu24a}*{Lemma 2.14}, the sequence $(\Psi_{j, m}^{\dist_Y^p})_{j\in \N}$ also satisfies the same bound. Thus 
\begin{align}\label{eqn: fatou bound}
\begin{split}
- 
\sum_{j\in \N}\chi_j(\omega)\zeta(\omega)\left(\int_Y (\Psi_{j, m}^{\dist_Y^p})_\omega d(\Xi^{-1}_{j, \omega})_\sharp\mathfrak{m}^\omega
+\int_Y(\Psi_{j, m})_\omega d(\Xi^{-1}_{j, \omega})_\sharp\mathfrak{n}^\omega\right)
&\geq -4\ell_0\zeta(\omega),
\end{split}
\end{align}
for each $\omega\in \Omega$. Also  by Lemma~\ref{lem: c_p transform of X_p}  
and the local finiteness of the $\chi_j$, we have that
\begin{align}
\begin{split}
 \label{eqn: continuous integrals converge}
&
\lim_{m\to\infty}\sum_{j\in \N}\chi_j\zeta\cdot\left(\int_Y (\Psi_{j, m}^{\dist_Y^p})_\bullet d(\Xi^{-1}_{j, \bullet})_\sharp\mathfrak{m}^\bullet
 +\int_Y(\Psi_{j, m})_\bullet d(\Xi^{-1}_{j, \bullet})_\sharp\mathfrak{n}^\bullet\right)\\
&=
\sum_{j\in \N}\chi_j\zeta\cdot\left(\int_Y [-T_{\ell_0}(-f^j_\bullet)]^{\dist_Y^p}(t)d(\Xi^{-1}_{j, \bullet})_\sharp\mathfrak{m}^\bullet(t)+\int_Y [-T_{\ell_0}(-f^j_\bullet(s))]d(\Xi^{-1}_{j, \bullet})_\sharp\mathfrak{n}^\bullet(s)\right),
\end{split}
 \end{align} 
holds $\sigma$-a.e.  
Since $C_b(\Omega; \mathcal{Y}_p)\subset L^0(\sigma; \mathcal{Y}_p)$  by Remark~\ref{V58}, 
all functions involved can be integrated against $\sigma$ again by Lemma~\ref{lem: measurability of transform}; by~\eqref{eqn: fatou bound} and since $\zeta\in L^{r'}(\sigma)\subset L^1(\sigma)$ we may apply Fatou's lemma, thus combining with \eqref{eqn: Lq/p norm dual estimate 2} and~\eqref{eqn: continuous integrals converge} we have
\begin{align*}
 &\varliminf_{m\to\infty}\left[-\sum_{j\in \N}\int_\Omega\chi_j\zeta\cdot\left(\int_Y (\Psi_{j, m}^{\dist_Y^p})_\bullet d(\Xi^{-1}_{j, \bullet})_\sharp\mathfrak{m}^\bullet
     -\int_Y (\Psi_{j, m})_\bullet d(\Xi^{-1}_{j, \bullet})_\sharp\mathfrak{n}^\bullet \right) d\sigma\right]\\
     &\geq-\int_\Omega\zeta\cdot\varliminf_{m\to\infty}\left(\int_E \sum_{j\in \N}\chi_j(\Psi_{j, m}^{\dist_Y^p})_\bullet\circ \Xi^{-1}_{j, \bullet} d\mathfrak{m}^\bullet
     +\int_E \sum_{j\in \N}\chi_j(\Psi_{j, m})_\bullet \circ\Xi^{-1}_{j, \bullet}d\mathfrak{n}^\bullet \right) d\sigma\\
&> \scrmk_{p, q}(\mathfrak{m}, \mathfrak{n})^p-2\varepsilon.
\end{align*}
Let 
\begin{align}\label{eqn: dual potential def}
    \Phi(u)&\coloneqq\sum_{j\in \N}\chi_j(\pi(u))\cdot(\Psi_{j, m}^{\dist_Y^p})_{\pi(u)}(\Xi^{-1}_{j, \pi(u)}(u)),\quad
    \Psi(v)\coloneqq\sum_{j\in \N}\chi_j(\pi(v))\cdot(\Psi_{j, m})_{\pi(v)}(\Xi^{-1}_{j, \pi(v)}(v)),
\end{align}
for an $m$ sufficiently large, then since $\mathfrak{m}^\omega$, $\mathfrak{n}^\omega$ are supported in $\pi^{-1}(\{\omega\})$ for each $\omega\in \Omega$, we have
\begin{align}
\begin{split}
\label{eqn: almost min}
 -\left(\int_E\zeta\Phi d\mathfrak{m}
     +\int_E\zeta\Psi d\mathfrak{n}\right)
    &=-\sum_{j\in \N}\int_\Omega\chi_j\zeta\cdot\left(\int_Y (\Psi_{j, m}^{\dist_Y^p})_\bullet d(\Xi^{-1}_{j, \bullet})_\sharp\mathfrak{m}^\bullet
     -\int_Y (\Psi_{j, m})_\bullet d(\Xi^{-1}_{j, \bullet})_\sharp\mathfrak{n}^\bullet \right) d\sigma\\
&> \scrmk_{p, q}(\mathfrak{m}, \mathfrak{n})^p-3\varepsilon.
\end{split}
\end{align}
As $\varepsilon>0$ is arbitrary, we will obtain the first equality in Theorem~\ref{thm: main disint}~\eqref{thm: disint duality} when $\scrmk_{p, q}(\mathfrak{m}, \mathfrak{n})<\infty$, if we can verify that $(\Phi, \Psi)\in \mathcal{A}_{p,E,\sigma}$. First, let $(v_n)_{n\in \N}$ be a sequence in $E$ converging to some $v_\infty\in E$. Then by the local finiteness of $\{U_j\}_{j\in \N}$, there is a finite set $J\subset \N$ such that 
\begin{align*}
    \{\pi(v_n)\}_{n\in \N}\cup \{\pi(v_\infty)\}\subset \bigcup_{j\in J}U_j.
\end{align*}
Hence
\begin{align*}
    \left| \Psi(v_n)-\Psi(v_\infty)\right|
   &\leq\sum_{j\in J}\left(\left| \chi_j(\pi(v_n))-\chi_j(\pi(v_\infty))\right|\cdot\left|(\Psi_{j, m})_{\pi(v_n)}(\Xi^{-1}_{j, \pi(v_n)}(v_n))\right|\right.\\
    &+\left| \chi_j(\pi(v_\infty))\right|\cdot\left|(\Psi_{j, m})_{\pi(v_n)}(\Xi^{-1}_{j, \pi(v_n)}(v_n))-(\Psi_{j, m})_{\pi(v_\infty)}(\Xi^{-1}_{j, \pi(v_n)}(v_n))\right|\\
    &+\left| \chi_j(\pi(v_\infty))\right|\cdot\left|(\Psi_{j, m})_{\pi(v_\infty)}(\Xi^{-1}_{j, \pi(v_n)}(v_n))-(\Psi_{j, m})_{\pi(v_\infty)}(\Xi^{-1}_{j, \pi(v_\infty)}(v_n))\right|\\
    &\left.+\left| \chi_j(\pi(v_\infty))\right|\cdot\left|(\Psi_{j, m})_{\pi(v_\infty)}(\Xi^{-1}_{j, \pi(v_\infty)}(v_n))-(\Psi_{j, m})_{\pi(v_\infty)}(\Xi^{-1}_{j, \pi(v_\infty)}(v_\infty))\right|\right)\\
    &\leq \sum_{j\in J}(I_{j, n}+II_{j, n}+III_{j, n}+IV_{j, n}),
\end{align*}
where
\begin{align*}
I_{j, n}
&\coloneqq 2\ell_0
\left| \chi_j(\pi(v_n))-\chi_j(\pi(v_\infty))\right|,\\
II_{j, n}
&\coloneqq \left\|(\Psi_{j, m})_{\pi(v_n)}-(\Psi_{j, m})_{\pi(v_\infty)}\right\|_{\mathcal{Y}_p}(1+\dist_{y_0}(\Xi^{-1}_{j, \pi(v_n)}(v_n))^p),\\
III_{j, n}
&\coloneqq \left|(\Psi_{j, m})_{\pi(v_\infty)}(\Xi^{-1}_{j, \pi(v_n)}(v_n))-(\Psi_{j, m})_{\pi(v_\infty)}(\Xi^{-1}_{j, \pi(v_\infty)}(v_n))\right|,\\
IV_{j, n}
&\coloneqq \left|(\Psi_{j, m})_{\pi(v_\infty)}(\Xi^{-1}_{j, \pi(v_\infty)}(v_n))-(\Psi_{j, m})_{\pi(v_\infty)}(\Xi^{-1}_{j, \pi(v_\infty)}(v_\infty))\right|.
\end{align*}

By continuity of the $\chi_j$, $\pi$, $\Xi^{-1}_{j, \pi(v_\infty)}$, and $(\Psi_{j, m})_{\pi(v_\infty)}$, we see 
\begin{align*}
\lim_{n\to\infty}(I_{j, n}+IV_{j, n})=0
\end{align*}
for each $j\in J$. Since $(v_n)_{n\in\mathbb{N}}$ is a convergent sequence,
\begin{align*}
    \dist_{y_0}(\Xi^{-1}_{j, \pi(v_n)}(v_n))
    &=\dist_E(\Xi_j(\pi(v_n), y_0), v_n)
\end{align*}
is bounded uniformly in $n$ by the continuity of the $\Xi_j$ and $\pi$, then combining with the fact that $(\Psi_{j, m})_\bullet\in C_b(\Omega; \mathcal{Y}_p)$ we see 
\begin{align*}
\lim_{n\to\infty}II_{j, n}=0
\end{align*}
for each $j\in J$. Also, 
\begin{align*}
\dist_Y(\Xi^{-1}_{j, \pi(v_n)}(v_n), \Xi^{-1}_{j, \pi(v_\infty)}(v_n))=\dist_E(v_n, \Xi_j(\pi(v_n), \Xi^{-1}_{j, \pi(v_\infty)}(v_n)))\xrightarrow{n\to\infty} 0
\end{align*}
by the continuity of $\pi$, $\Xi_j$, and $\Xi^{-1}_{j, \pi(v_\infty)}$, hence 
\begin{align*}
\lim_{n\to\infty} III_{j, n}=0.
\end{align*}
Again by the local finiteness of the family $\{\chi_j\}_{j\in \N}$, the sum in the bound for $\left| \Psi(v_n)-\Psi(v_\infty)\right|$ is actually finite, hence we see $\Psi\in C(E)$. Since Lemma~\ref{lem: c_p transform of X_p}~\eqref{eqn: Xp transform conv}  implies $(\Psi_j^{\dist_Y^p})_\bullet$ is continuous with respect to $\left\|\cdot\right\|_{\mathcal{Y}_p}$, a similar argument shows $\Phi\in C(E)$, and the uniform boundedness of the $(\Psi_{j, m})_\bullet$ implies $\Phi$, $\Psi\in C_b(E)$. Finally, if $\omega\coloneqq \pi(u)=\pi(v)$, then
\begin{align*}
    -\Phi(u)-\Psi(v)&=\sum_{j\in \N}\chi_j(\omega)(-(\Psi_{j, m}^{\dist_Y^p})_{\omega}\circ\Xi^{-1}_{j, \omega}(u)-(\Psi_{j, m})_{\omega}\circ\Xi^{-1}_{j, \omega}(v))\\
    &\leq \sum_{j\in \N}\chi_j(\omega)\dist_Y(\Xi^{-1}_{j, \omega}(u), \Xi^{-1}_{j, \omega}(v))^p
    =\dist_E(u, v)^p,
\end{align*}
thus $(\Phi, \Psi)\in \mathcal{A}_{p,E,\sigma}$ as desired.

If $\scrmk_{p, q}(\mathfrak{m}, \mathfrak{n})=\infty$, we can replace $\scrmk_{p, q}(\mathfrak{m}, \mathfrak{n})$ in the above proof starting at~\eqref{eqn: disint dual inequality} by an arbitrary positive number to obtain that the supremum in the first equality of Theorem~\ref{thm: main disint}~\eqref{thm: disint duality} takes the value~$\infty$.

Now let us assume that $(E, \dist_E)$ is locally compact. To show the second equality in Theorem~\ref{thm: main disint}~\eqref{thm: disint duality}, fix $\varepsilon>0$ and take $(\Phi, \Psi)\in \mathcal{A}_{p,E,\sigma}$ defined by~\eqref{eqn: dual potential def}, satisfying~\eqref{eqn: almost min} as above. By definition of $S_p$ and since $\mathfrak{m}^\omega,\mathfrak{n}^\omega$ are supported on $\pi^{-1}(\{\omega\})$ we see that for $\sigma$-a.e. $\omega$, 
\begin{align*}
    -\int_E\Phi d\mathfrak{m}^\omega
     -\int_E\Psi d\mathfrak{n}^\omega
     &\leq -\int_ES_p\Psi d\mathfrak{m}^\omega
     -\int_E\Psi d\mathfrak{n}^\omega
     \leq \mk_p^E(\mathfrak{m}^\omega, \mathfrak{n}^\omega)^p,
\end{align*}
hence 
\begin{align*}
     -\int_\Omega\zeta\cdot\left(\int_ES_p\Psi d\mathfrak{m}^\bullet
     +\int_E\Psi d\mathfrak{n}^\bullet\right)d\sigma
     =&-\int_E\zeta S_p\Psi d\mathfrak{m}
     -\int_E\zeta\Psi d\mathfrak{n}\\
     &\in \left(\scrmk_{p, q}(\mathfrak{m}, \mathfrak{n})^p-3\varepsilon, \scrmk_{p, q}(\mathfrak{m}, \mathfrak{n})^p\right].
\end{align*}
Since $\Phi$ and $\Psi$ are uniformly bounded from below, we can view 
\begin{align*}
-\zeta\cdot\left(
 \int_E
\Psi
d\mathfrak{n}^\bullet+
\int_E
S_p\Psi d\mathfrak{m}^\bullet \right) \sigma
\end{align*}
as a (signed) Borel measure with finite total variation on $\Omega$, then from~\cite{Bogachev07}*{Theorem 7.1.7} we can find a compact set $K'_\varepsilon\subset \Omega$ such that
\begin{align}\label{eqn: outside integral small}
   \left| -\int_{\Omega\setminus K'_\varepsilon}
\zeta\cdot\left(
 \int_E
\Psi
d\mathfrak{n}^\bullet+
\int_E
S_p\Psi d\mathfrak{m}^\bullet \right) d\sigma\right|<\frac{\varepsilon}{2}.
\end{align}
Since $\Omega$ is locally compact, we may cover $K'_\varepsilon$ with a finite number of open sets whose closures are compact. Writing $K_\varepsilon$ for the union of the closures of these neighborhoods, we see $K_\varepsilon$ is also compact and~\eqref{eqn: outside integral small} holds with $K'_\varepsilon$ replaced by $K_\varepsilon^\circ$. 
Now define for $\delta>0$
\begin{align*}
\psi_{\delta, \varepsilon}(\omega)&\coloneqq\min\{1, \delta^{-1}\dist_\Omega(\omega,\Omega\setminus K_\varepsilon)\},\qquad
    \xi_{\delta, \varepsilon}(v)\coloneqq \psi_{\delta, \varepsilon}(\pi(v))\cdot \Psi(v).
\end{align*}
Since $\Psi$ is bounded on $E$ by $2\ell_0$, so is $S_p\Psi$, hence for any $u\in E$ and $\tilde{\varepsilon}>0$ there exists some $v_{\tilde{\varepsilon}}\in \pi^{-1}(\{\pi(u)\})$ such that $S_p\Psi(u)\leq -\dist_E(u, v_{\tilde{\varepsilon}})^p-\Psi(v_{\tilde{\varepsilon}})+\tilde{\varepsilon}$. Thus 
\begin{align*}
    S_p\Psi(u)-S_p\xi_{\delta, \varepsilon}(u)
    &\leq -\dist_E(u, v_{\tilde{\varepsilon}})^p-\Psi(v_{\tilde{\varepsilon}})+\tilde{\varepsilon}+\inf_{v\in \pi^{-1}(\{\pi(u)\})}(\dist_E(u, v)^p+\xi_{\delta, \varepsilon}(v))\\
    &\leq \xi_{\delta, \varepsilon}(v_{\tilde{\varepsilon}})-\Psi(v_{\tilde{\varepsilon}})+\tilde{\varepsilon}\\&\leq 2\ell_0(\psi_{\delta, \varepsilon}(\pi(v_{\tilde{\varepsilon}}))-1)+\tilde{\varepsilon}
    =2\ell_0(\psi_{\delta, \varepsilon}(\pi(u))-1)+\tilde{\varepsilon}.
\end{align*}
Taking $\tilde{\varepsilon}\to 0$ and by an analogous argument reversing the roles of $\Psi$ and~$\xi_{\delta, \varepsilon}$, we obtain
\begin{align*}
    \left| \int_{K_\varepsilon^\circ}
\zeta\cdot\left(
 \int_E S_p\xi_{\delta, \varepsilon} d\mathfrak{m}^\bullet-
\int_E
S_p\Psi d\mathfrak{m}^\bullet \right) d\sigma\right|
&\leq
 2\ell_0\left| \int_{K_\varepsilon^\circ}
\zeta\left| 1-\psi_{\delta, \varepsilon}\right| d\sigma\right|\\
&\leq
2\ell_0\left\|\zeta\mathds{1}_{\{\omega\in K_\varepsilon^\circ\mid 0\leq\dist_\Omega(\omega, \Omega\setminus K_\varepsilon)<\delta\}}\right\|_{L^{r'}(\sigma)}.
\end{align*}
We also find
\begin{align*}
     \left| \int_{K_\varepsilon^\circ}
\zeta\cdot\left(
 \int_E \xi_{\delta, \varepsilon} d\mathfrak{n}^\bullet-
\int_E
\Psi d\mathfrak{n}^\bullet \right) d\sigma\right|
&\leq 
2\ell_0\left\|\zeta\mathds{1}_{\{\omega\in K_\varepsilon^\circ\mid 0\leq\dist_\Omega(\omega, \Omega\setminus K_\varepsilon)<\delta\}}\right\|_{L^{r'}(\sigma)},
\end{align*}
thus if $\delta>0$ is sufficiently small, combining with~\eqref{eqn: outside integral small} and using the definition of $S_p$ implies that
\begin{align*}
    -\int_\Omega
\zeta\cdot\left(
 \int_E
\xi_{\delta, \varepsilon}
d\mathfrak{n}^\bullet+
\int_E
S_p\xi_{\delta, \varepsilon} d\mathfrak{m}^\bullet \right) d\sigma
\in (\scrmk_{p, q}(\mathfrak{m}, \mathfrak{n})^p-4\varepsilon, \scrmk_{p, q}(\mathfrak{m}, \mathfrak{n})^p].
\end{align*}
Since $\varepsilon$ is arbitrary, we need only verify that $\xi_{\delta, \varepsilon}\in \mathcal{X}_p$; note it is clear that $\xi_{\delta, \varepsilon}\in C_b(E)$.

Now since $\{U_j\}_{j\in \N}$ is locally finite, the compact set $K_\varepsilon$ can only intersect a finite number of  sets $\{U_{j_i}\}_{i=1}^{J_I}$. Thus for any fixed $\hat{\varepsilon}>0$, using that $\xi_{\delta, \varepsilon}\equiv 0$ outside of $\pi^{-1}(K_\varepsilon)$,
\begin{align}
\begin{split}
\label{eqn: Zp verify}
    &\left\{v\in E \biggm| \frac{1}{1+\dist_{E,y_0}^p(\pi(v),v)}\left|\xi_{\delta, \varepsilon}(v)\right|\geq \hat{\varepsilon}\right\}\\
    \subset& \left\{v\in \pi^{-1}(K_\varepsilon) \biggm| \sum_{i=1}^I \dfrac{\chi_{j_i}(\pi(v))}{1+\dist_{E,y_0}^p(\pi(v),v)}\psi_{\delta, \varepsilon}(\pi(v))\cdot\left|(\Psi_{j_i, m})_{\pi(v)}(\Xi^{-1}_{j_i, \pi(v)}(v))\right|\geq \hat{\varepsilon}\right\}\\
    \subset& \bigcup_{i=1}^IA_{i},
\end{split}
\end{align}
where
\[
A_i\coloneqq \left\{v\in \pi^{-1}(K_\varepsilon) \biggm| \frac{\chi_{j_i}(\pi(v))}
{1+\dist_{E,y_0}^p(\pi(v),v)}
\left|(\Psi_{j_i, m})_{\pi(v)}(\Xi^{-1}_{j_i, \pi(v)}(v))\right|
\geq \frac{\hat{\varepsilon}}{I}\right\}.
\]
For $1\leq i\leq I$ fixed, let $(v_\ell)_{\ell\in \N}$ be a sequence in $A_{i}$. Then if $\omega_\ell\coloneqq \pi(v_\ell)$, by compactness of $K_\varepsilon$ there exists a subsequence such that $\omega_\ell$ converges to some $\omega_\infty\in K_\varepsilon$. Also since $\chi_{j_i}(\pi(v_\ell))>0$ we have $\omega_\ell\in U_{j_i}$, hence we may define $y_\ell\coloneqq \Xi^{-1}_{j_i, \omega_\ell}(v_\ell)$. Then we have
\begin{align*}
    \chi_{j_i}(\omega_\ell)\left|(\Psi_{j_i, m})_{\omega_\ell}(y_\ell)\right|
    \geq \frac{\hat{\varepsilon}}{I}\left(1+\dist_{E,y_0}^p(\omega_\ell, v_\ell)\right)
    > \frac{\hat{\varepsilon}}{I}\left(\chi_{j_i}(\omega_\ell)+\chi_{j_i}(\omega_\ell)\dist_{y_0}(y_\ell)^p\right),
\end{align*}
since we must have $\chi_{j_i}(\omega_\ell)>0$, this implies
\begin{align*}
    \frac{1}{1+\dist_{y_0}(y_\ell)^p}\left|(\Psi_{j_i, m})_{\omega_\infty}(y_\ell)\right|
    &\geq \frac{\left|(\Psi_{j_i, m})_{\omega_\ell}(y_\ell)\right|}{1+\dist_{y_0}(y_\ell)^p}-\frac{\left|(\Psi_{j_i, m})_{\omega_\ell}(y_\ell)-(\Psi_{j_i, m})_{\omega_\infty}(y_\ell)\right|}{1+\dist_{y_0}(y_\ell)^p}\\
    &\geq \frac{\tilde{\varepsilon}}{I}-\left\| (\Psi_{j_i, m})_{\omega_\ell}-(\Psi_{j_i, m})_{\omega_\infty}\right\|_{\mathcal{Y}_p}\\
    &\geq\frac{\tilde{\varepsilon}}{2I}
\end{align*}
if $\ell$ is large enough. Since $(\Psi_{j_i, m})_{\omega_\infty}\in \mathcal{Y}_p$ there exists a subsequence of $y_\ell$ converging to some $y_\infty\in Y$. Thus by continuity of $\Xi_{j_i}$, we see (the corresponding subsequence of) $v_\ell$ converges to $v_\infty\coloneqq \Xi_{j_i}(\omega_\infty, y_\infty)$ which we easily see belongs to $A_{i}$. Thus as a closed subset of a finite union of compact sets, the first set in~\eqref{eqn: Zp verify} is compact, in particular we see $\xi_{\delta, \varepsilon}\in \mathcal{X}_p$, finishing the proof.
\end{proof}
\section{Further properties of disintegrated Monge--Kantorovich metrics}\label{sec: disint etc}
In this section, we prove some further properties of the metrics $\scrmk_{p, q}$. 
First, we prove that convergence in $\scrmk_{p, q}$ implies weak convergence.
\begin{proposition}\label{prop: disint conv implies weak}
    For any $1\leq p<\infty$ and $1\leq q\leq \infty$, if the sequence $(\mathfrak{m}_\ell)_{\ell\in \N}$ in $\mathcal{P}^\sigma_{p, q}(E)$ converges in $\scrmk_{p, q}$ to some $\mathfrak{m}\in \mathcal{P}^\sigma_{p, q}(E)$, then the sequence converges weakly.
\end{proposition}
\begin{proof}
Any subsequence of $(\mathfrak{m}_\ell)_{\ell\in \N}$ has a further subsequence (not relabeled) such that 
the sequence $(\mk^E_p(\mathfrak{m}^\omega_\ell, \mathfrak{m}^\omega))_{\ell\in \mathbb{N}}$ converges to zero for $\sigma$-a.e. $\omega$. Then for any $\phi\in C_b(E)$, by Theorem~\ref{thm: wassconv} we have 
    \begin{align*}
    \lim_{\ell\to\infty}\int_E\phi d\mathfrak{m}^\omega_\ell=\int_E\phi d\mathfrak{m}^\omega,
    \end{align*}
    then by dominated convergence, 
    \begin{align*}
    \lim_{\ell\to\infty}\int_E\phi d\mathfrak{m}_\ell=\int_E\phi d\mathfrak{m}.
    \end{align*}
    Since this holds for any subsequences, we have weak convergence of the whole sequence to $\mathfrak{m}$.
\end{proof}
Next, duality will yield that $\scrmk_{p, q}$ is lower semi-continuous with respect to weak convergence on ~$\mathcal{P}^\sigma_p(E)$, at least when $E$ is locally compact.
\begin{corollary}\label{cor: disint lsc}
    If $(E, \dist_E)$ is locally compact, $p\leq q$,  and $(\mathfrak{m}_\ell)_{\ell\in \N}$ and $(\mathfrak{n}_\ell)_{\ell\in \N}$ are sequences in~$\mathcal{P}^\sigma_p(E)$ that weakly converge to $\mathfrak{m}$ and $\mathfrak{n}\in \mathcal{P}^\sigma_p(E)$ respectively, then
    \begin{align*}
        \scrmk_{p, q}(\mathfrak{m}, \mathfrak{n})\leq \varliminf_{\ell\to\infty}\scrmk_{p, q}(\mathfrak{m}_\ell, \mathfrak{n}_\ell).
    \end{align*}
\end{corollary}
\begin{proof}
    Fix $\zeta\in \mathcal{Z}_{r', \sigma}$ and $(\Phi, \Psi)\in \mathcal{A}_{p,E,\sigma}$, then since $(\zeta\circ \pi)\Phi$, $(\zeta\circ \pi)\Psi\in C_b(E)$ we have
\begin{align*}
\left(-\int_E(\zeta\circ \pi)\Phi d\mathfrak{m}-\int_E(\zeta\circ \pi)\Psi d\mathfrak{n}\right)^{\frac{1}{p}}
&=\lim_{\ell\to\infty} \left(-\int_E(\zeta\circ \pi)\Phi d\mathfrak{m}_\ell-\int_E(\zeta\circ \pi)\Psi d\mathfrak{n}_\ell\right)^{\frac{1}{p}}\\
    &\leq \varliminf_{\ell\to\infty} \scrmk_{p, q}(\mathfrak{m}_\ell, \mathfrak{n}_\ell),
\end{align*}
where we have used Theorem~\ref{thm: main disint}~\eqref{thm: disint duality} in the last line. Taking a supremum over $\zeta\in \mathcal{Z}_{r', \sigma}$ and $(\Phi, \Psi)\in \mathcal{A}_{p,E,\sigma}$ and using Theorem~\ref{thm: main disint}~\eqref{thm: disint duality} again yields the desired lower semi-continuous.
\end{proof}
Now we show that $\scrmk_{p, p}$ arises from a certain optimal transport problem on $E^2$.
\begin{definition}
For $1\leq p<\infty$, define 
$\mathfrak{c}_p: E^2 \to [0,\infty]$ by 
\[
\mathfrak{c}_p(u, v)
\coloneqq  \begin{cases}
 \dist_E(u,v)^p,& \text{if }\pi(u)=\pi(v),\\
 \infty, &\text{else}.
 \end{cases}
\]
For $\mathfrak{m}$,  $\mathfrak{n}\in\mathcal{P}^\sigma_p(E)$,
set 
\[
\mathfrak{C}_p(\mathfrak{m}, \mathfrak{n})
\coloneqq \inf_{\Gamma\in \Pi(\mathfrak{m}, \mathfrak{n})}
\| \mathfrak{c}_p\|_{L^{p}(\Gamma)}\in[0,\infty].
\]
\end{definition}
\begin{proposition}\label{OT}
For $\mathfrak{m}$, $\mathfrak{n}\in\mathcal{P}^\sigma_{p, p}(E)$,
$\mathfrak{C}_p(\mathfrak{m}, \mathfrak{n})$ is finite and 
\[
\mathfrak{C}_p(\mathfrak{m}, \mathfrak{n})
=\scrmk_{p, p}(\mathfrak{m}, \mathfrak{n})^p.
\]
\end{proposition}
\begin{proof}
Fix $\mathfrak{m}$, $\mathfrak{n}\in \mathcal{P}^\sigma_{p, p}(E)$.
For any $(\Phi, \Psi)\in \mathcal{A}_{p,E,\sigma}$,  by definition we have 
$\Phi$, $\Psi\in C_b(E)$ and  
    $-\Phi(u)-\Psi(v)\leq\mathfrak{c}_p(u,v)$.
Since $(E, \dist_E)$ is a complete, separable metric space, the Kantorovich duality Theorem~\ref{Kantorovich} (we have stated Theorem~\ref{Kantorovich} only for cost functions of the form $\dist_Y^p$, however the same result holds for any lower semi-continuous cost function bounded from below, hence for $\mathfrak{c}_p$, see~\cite{Villani09}*{Theorem 5.10}) yields
\begin{align*}
\mathfrak{C}_p(\mathfrak{m},\mathfrak{n}) 
=\sup_{(\Phi,\Psi)\in \mathcal{A}_{p,E,\sigma}}
\left(
-\int_E\Phi d\mathfrak{m}
-\int_E\Psi d\mathfrak{n}
\right)
&=\sup_{(\Phi,\Psi)\in \mathcal{A}_{p,E,\sigma}}
\int_{\Omega}
\left(
-\int_E\Phi d\mathfrak{m}^{\bullet}
-\int_E\Psi d\mathfrak{n}^{\bullet}
\right)
d\sigma\\
&\leq 
\int_\Omega
\mk_p^E(\mathfrak{m}^{\omega},\mathfrak{n}^{\omega})^p
d\sigma(\omega)
=\scrmk_{p,p}(\mathfrak{m},\mathfrak{n})^p\\
&<\infty.
\end{align*}
Thus
$\mathfrak{C}_p(\mathfrak{m},\mathfrak{n})$ is finite and
$\mathfrak{C}_p(\mathfrak{m},\mathfrak{n})\leq \scrmk_{p,p}(\mathfrak{m},\mathfrak{n})^p$.

On the other hand, 
since $\mathfrak{c}_p$ is lower semi-continuous and non-negative, 
by \cite{Villani09}*{Theorem~4.1}
there exists $\gamma\in \Pi(\mathfrak{m},\mathfrak{n})$ such that 
\[
\mathfrak{C}_p (\mathfrak{m},\mathfrak{n})
=\int_{E^2}
\mathfrak{c}_p d\gamma,
\]
since $\mathfrak{C}_p (\mathfrak{m},\mathfrak{n})<\infty$ by above, we find that
\begin{align*}
 \gamma(\{ (u, v)\mid \pi(u)\neq\pi(v)\})=0.
\end{align*}
Let $\pi^2: 
E^2 
\to \Omega^2
$
be defined by $\pi^2(u, v)\coloneqq (\pi(u), \pi(v))$, 
then by the above, for $\mathcal{B}_\sigma$-measurable sets $A$, $A'\subset \Omega$ we have
\begin{align*}
 \pi^2
_\sharp\gamma(A\times A')
&=\gamma(\{(u, v)\mid\pi(u)\in A,\ \pi(v)\in A',\ \pi(u)=\pi(v)\})\\
&=\gamma(\{(u, v)\mid\pi(u), \pi(v)\in A\cap A'\})
=\gamma(\pi^{-1}(A\cap A')\times E)\\
&=\mathfrak{m}(\pi^{-1}(A\cap A'))
 =\sigma(A\cap A')
 =(\Id_\Omega\times \Id_\Omega)_\sharp\sigma(A\times A'),
\end{align*}
hence
$\pi^2_\sharp\gamma=(\Id_\Omega\times \Id_\Omega)_\sharp\sigma$.
Consider the disintegration of $\gamma$ with respect to ~$\pi^2$ given by
\[
\gamma
=\gamma^{(\bullet,\ast)}\otimes \pi^2_\sharp\gamma
=\gamma^{(\bullet,\ast)}\otimes (\Id_\Omega\times \Id_\Omega)_\sharp\sigma.
\]
For $\phi\in C_b(E^2)$,
the function on $\Omega^2$ (resp.\,$\Omega$) defined by
 \[
(\omega, \omega') \mapsto \int_{E^2}\phi d\gamma^{(\omega, \omega')}
\qquad
\left(
\text{resp.}\ 
\omega \mapsto \int_{E^2}\phi d\gamma^{(\omega, \omega)}
\right)
 \]
is Borel by \nameref{thm: disintegration},
and 
\begin{align}\label{disintegration}
\int_{\Omega^2} \int_{E^2}\phi d\gamma^{(\omega, \omega')} d\pi^2_\sharp\gamma(\omega,\omega')
=
\int_\Omega \int_{E^2}\phi d\gamma^{(\omega, \omega)} d\sigma(\omega).
\end{align}

Now for any Borel set $E'\subset E$ and $\Omega'\in \mathcal{B}_\sigma$, since $\gamma\in \Pi(\mathfrak{m}, \mathfrak{n})$ we have
\begin{align*}
\int_{\Omega'}
\mathfrak{m}^\bullet(E')d\sigma
&=\int_\Omega\int_E\mathds{1}_{\Omega'}(\pi(u))
\mathds{1}_{E'}(u) d\mathfrak{m}^\bullet(u)d\sigma
=\int_E\mathds{1}_{\Omega'}(\pi(u))
\mathds{1}_{E'}(u) d\mathfrak{m}(u)\\
&=
\int_{ E^2}\mathds{1}_{\Omega'}(\pi(u))
\mathds{1}_{E'}(u) d\gamma(u, v) 
=
\int_{\Omega}
\int_{E^2}\mathds{1}_{\Omega'}(\pi(u))
\mathds{1}_{E'}(u) d\gamma^{(\omega, \omega)}(u, v)
d\sigma(\omega)\\
&=\int_{\Omega'}
\int_{E^2}
\mathds{1}_{E'\times E}(u, v) d\gamma^{(\omega, \omega)}(u, v)
d\sigma(\omega)
=\int_{\Omega'}
\gamma^{(\omega, \omega)}(E'\times E)
d\sigma(\omega).
\end{align*}
Since $E'$ and $\Omega'$ are arbitrary (and using a similar argument for $\mathfrak{n}$) this implies that for $\sigma$-a.e.~$\omega\in \Omega$, we have 
$
\gamma^{(\omega, \omega)}
\in \Pi(\mathfrak{m}^\omega,\mathfrak{n}^\omega)$.

Finally, using this claim with the disintegration~\eqref{disintegration}, we have
\begin{align*}
\scrmk_{p, p}(\mathfrak{m}, \mathfrak{n})^p
&=
 \int_\Omega
\mk_p^E(\mathfrak{m}^{\omega}, \mathfrak{n}^\omega)^pd\sigma(\omega)\\
&\leq \int_\Omega\int_{E^2} \dist_E(u,v)^pd \gamma^{(\omega, \omega)}(u, v)d\sigma(\omega)
=\int_\Omega\int_{E^2} \mathfrak{c}_p( u, v )d \gamma^{(\omega, \omega)}(u, v)d\sigma(\omega)\\
&=
\int_{E^2} \mathfrak{c}_p(u, v)d\gamma (u,v)=\mathfrak{C}_p(\mathfrak{m}, \mathfrak{n}),
\end{align*}
completing the proof of the lemma.
\end{proof}
We also show that in the case of a trivial bundle where the fiber equals the base space, the set of $p$-optimal couplings is closed in $\scrmk_{p, q}$ for $p\leq q$.
\begin{proposition}\label{prop: optimal closed}
    Suppose $(\Omega, \dist_\Omega)$ is a complete, separable metric space, we have the trivial bundle $E=\Omega\times \Omega$. Fix $1\leq p<\infty$ and some $\sigma\in \mathcal{P}_p(\Omega)$, and let us denote by $\Pi_{\mathrm{opt}}(\sigma)$ the set of all $p$-optimal couplings between $\sigma$ and any other measure in $\mathcal{P}_p(\Omega)$. Then if $p\leq q\leq \infty$, the set $\Pi_{\mathrm{opt}}(\sigma)$  is closed with respect to $\scrmk_{p, q}$ in $\mathcal{P}^\sigma_{p, q}(\Omega\times \Omega)$.
\end{proposition}
\begin{proof}
Let $(\mu_\ell)_{\ell\in \N}\subset \mathcal{P}_p(\Omega)$ and suppose $\gamma_\ell$ is a $p$-optimal coupling between $\mu_\ell$ and $\sigma$, note that $\gamma\in \mathcal{P}^\sigma(E)$. In the calculations below we will consider each $\gamma_\ell^\bullet$ as a measure on $\Omega$. Since $p\leq q<\infty$, for some $\omega_0\in \Omega$ we can calculate using Jensen's inequality that
\begin{align*}
    \scrmk_{p, q}(\delta^\bullet_{E, \omega_0}\otimes \sigma, \gamma_\ell)
    &=\lVert \mk^\Omega_p(\delta^\Omega_{\omega_0}, \gamma_\ell^\bullet)\rVert_{L^q(\sigma)}
    =\left\lVert \left(\int_\Omega \dist_\Omega(\omega_0, \omega)^pd\gamma_\ell^\bullet(\omega)\right)^{\frac{1}{p}}\right\rVert_{L^q(\sigma)}\\
    &\leq \left(\int_\Omega\int_\Omega \dist_\Omega(\omega_0, \omega)^pd\gamma_\ell^{\omega'}(\omega)d\sigma(\omega')\right)^{\frac{1}{p}}
    =\left(\int_{\Omega^2} \dist_\Omega(\omega_0, \omega)^pd\gamma_\ell(\omega', \omega)\right)^{\frac{1}{p}}\\
    &=\left(\int_{\Omega} \dist_\Omega(\omega_0, \omega)^pd\mu_\ell(\omega)\right)^{\frac{1}{p}}\\&<\infty.
\end{align*}
Taking $q\to\infty$ also yields that $\scrmk_{p, \infty}(\delta^\bullet_{E, \omega_0}\otimes \sigma, \gamma_\ell)<\infty$. 
Now suppose $(\gamma_\ell)_{\ell\in \mathbb{N}}$ converges in $\scrmk_{p, q}$ to some $\gamma\in \mathcal{P}^\sigma_{p, q}(\Omega\times\Omega)$. Again since $p\leq q$, by H\"older's inequality,
\begin{align*}
    \int_{\Omega^2} \dist_\Omega(\omega', \omega)^pd\gamma_\ell(\omega', \omega)
    &\leq 2^{p-1}\left(\int_\Omega \dist_\Omega(\omega_0, \omega')^pd\sigma(\omega')+\int_\Omega\int_\Omega \dist_\Omega(\omega_0, \omega)^pd\gamma_\ell^{\omega'}(\omega)d\sigma(\omega')\right)\\
    &\leq 2^{p-1}\left(\int_\Omega \dist_\Omega(\omega_0, \omega')^pd\sigma(\omega')+\left\lVert\int_\Omega \dist_\Omega(\omega_0, \omega)^pd\gamma_\ell^{\bullet}(\omega)\right\rVert_{L^{q/p}(\sigma)}\right)\\
    &=2^{p-1}\left(\int_\Omega \dist_\Omega(\omega_0, \omega')^pd\sigma(\omega')+\scrmk_{p, q}(\delta^\bullet_{E, \omega_0}\otimes \sigma, \gamma_\ell)^p\right)
\end{align*}
which is bounded uniformly in $\ell$. By Proposition~\ref{prop: disint conv implies weak} the sequence converges weakly, hence by~\cite{Villani09}*{Theorem 5.20} we see $\gamma\in\Pi_{\mathrm{opt}}(\sigma)$ as well.
\end{proof}
Finally, we note there is also a relationship between the sliced Monge--Kantorovich metrics which we defined in our previous work~\cite{KitagawaTakatsu24a}, and our disintegrated Monge--Kantorovich metrics.
\begin{definition}[\cite{KitagawaTakatsu24a}*{Definition 1.1}]
    For $n\in \N$, let $\sigma_{n-1}$ be the standard Riemannian volume measure on $\S^{n-1}$, normalized to have unit mass, and for $\omega\in \S^{n-1}$ define the map $R^\omega: \R^n\to \R$ by $R^\omega(x)\coloneqq \langle x,\omega\rangle$. Then for $1\leq p<\infty$, $1\leq q\leq \infty$, and $\mu$, $\nu\in \mathcal{P}_p(\R^n)$, the \emph{sliced $(p, q)$-Monge--Kantorovich metric} is defined by
    \begin{align*}
        \mk_{p, q}(\mu, \nu)\coloneqq \left\| \mk_p^\R(R^\bullet_\sharp \mu, R^\bullet_\sharp \nu)\right\|_{L^q(\sigma_{n-1})}.
    \end{align*}
\end{definition}
Recall these include the well-known \emph{sliced Wasserstein} ($p=q$) and \emph{max-sliced Wasserstein} ($q=\infty$) metrics. As shown in~\cite{KitagawaTakatsu24a}*{Main Theorem}, each $(\mathcal{P}_p(\R^n), \mk_{p, q})$ is a complete, separable metric space, but is \emph{not} geodesic (when $p>1$). The relationship between the sliced and disintegrated Monge-Kantorovich metrics is as follows.
\begin{proposition}\label{thm: disint embed} 
Let $n\in\mathbb{N}$.
If $(E, \Omega, \pi, Y)$ is taken to be the trivial bundle $E=\S^{n-1}\times \R$, 
then there exists an isometric embedding of 
$(\mathcal{P}_p(\mathbb{R}^n), \mk_{p,q})$ into $(\mathcal{P}^{\sigma_{n-1}}_{p,q}(E), \scrmk[\sigma_{n-1}]_{p, q})$ 
defined by sending $\mu \in \mathcal{P}_p(\mathbb{R}^n)$ to the element of the form 
$\rad^\bullet_\sharp\mu\otimes \sigma_{n-1}$.
\end{proposition}
\begin{proof}
Let $\mu \in \mathcal{P}(\mathbb{R}^n)$.
For 
$\phi\in C_b(\S^{n-1}\times\R)$,
by dominated convergence the function on $\mathbb{S}^{n-1}$ defined by 
\[
\omega \mapsto 
\int_{\mathbb{R}} \phi(\omega, t)dR^\omega_\sharp\mu(t)
=
\int_{\mathbb{R}^n} \phi(\omega,\langle x,\omega\rangle)d\mu(x)
\]
is continuous, and
\begin{align*}
\mathcal{L}_\mu(\phi)
\coloneqq\int_{\mathbb{S}^{n-1}}\int_{\mathbb{R}} \phi(\omega, t)dR^\omega_\sharp\mu(t)d\sigma_{n-1}(\omega)
=\int_{\mathbb{S}^{n-1}}\int_{\mathbb{R}^n} 
\phi(\omega, \langle x,\omega\rangle)d\mu(x)d\sigma_{n-1}(\omega)
\end{align*}
is well-defined.
Since $\S^{n-1}\times\R$ is locally compact, by~\cite{Bogachev07}*{Theorem~7.11.3} we can identify $\mathcal{L}_\mu$ with a Borel probability measure $\mathfrak{m}_\mu\in \mathcal{P}^{\sigma_{n-1}}(\S^{n-1}\times\R)$ and $\mathfrak{m}_\mu^\bullet=R^\bullet_\sharp\mu$.

Noting that for the choice $y_0=0$ in $\R$, we have $\delta^\omega_{E, y_0}=\delta^\R_0$ for all $\omega\in \S^{n-1}$, for $\mu\in\mathcal{P}_p(\mathbb{R}^n)$ a direct calculation combined with \cite{KitagawaTakatsu24a}*{Lemma 2.3} gives
\begin{align*}
    \left\| \mk_p^\R(\delta_0^\R, \rad^\bullet_\sharp \mu)\right\|_{L^q(\sigma_{n-1})}
    &=\mk_{p,q}(\delta_0^{\mathbb{R}^n}, \mu)
    \leq M_{\max\{p, q\}, n}\mk_p^{\R^n}(\delta_0^{\R^n}, \mu)<\infty,
\end{align*}
hence 
$\mathfrak{m}_\mu\in \mathcal{P}^{\sigma_{n-1}}_{p, q}(\S^{n-1}\times\R)$.
Finally, for $\mu$, $\nu\in \mathcal{P}_p(\mathbb{R}^n)$,
we have
\begin{align*}
 \scrmk_{p,q}(\mathfrak{m}_\mu, \mathfrak{m}_\nu)
 &=\left\| \mk_p^{\R}(\mathfrak{m}_\mu^\bullet, \mathfrak{m}_\nu^\bullet)\right\|_{L^q(\sigma_{n-1})}
 =\left\| \mk_p^{\R}(R^\bullet_\sharp \mu, R^\bullet_\sharp \nu)\right\|_{L^q(\sigma_{n-1})}
 =\mk_{p,q}(\mu, \nu),
\end{align*}
showing that the map $\mu\mapsto \mathfrak{m}_\mu$ is an isometry.
\end{proof}
\begin{remark}
    By the completeness from \cite{KitagawaTakatsu24a}*{Main Theorem}, the image of $(\mathcal{P}_p(\mathbb{R}^n), \mk_{p,q})$ under $\mu\mapsto \mathfrak{m}_\mu$ is closed in $(\mathcal{P}^{\sigma_{n-1}}_{p, q}(\S^{n-1}\times\R), \scrmk[\sigma_{n-1}]_{p,q})$. However, also by \cite{KitagawaTakatsu24a}*{Main Theorem} the embedded image is \emph{not} geodesically convex in $(\mathcal{P}^{\sigma_{n-1}}_{p, q}(\S^{n-1}\times\R), \scrmk[\sigma_{n-1}]_{p,q})$ when $n\geq 2$ and $p>1$.

This shows that $(\mathcal{P}_p(\mathbb{R}^n), \mk_{p,q})$ can be viewed as a sort of ``submanifold'' embedded into the geodesic space $(\mathcal{P}^{\sigma_{n-1}}_{p,q}(\S^{n-1}\times\R), \scrmk[\sigma_{n-1}]_{p, q})$, but $\mk_{p,q}$ is in actuality utilizing the ambient metric from the larger space rather than the intrinsic metric generated from itself. In fact, it is proved in~\cite{Tilh}*{Lemma~2.6 and Lemma 2.8} that the intrinsic metric on $\mathcal{P}_p(\mathbb{R}^n)$ induced by $\mk_{p,p}$ between discrete measures with compact supports is $\mk_p^{\mathbb{R}^n}$.
\end{remark}
\begin{remark}
Recall that
$\mathcal{P}_2(\mathbb{R}^n)$ can be viewed as the quotient space of $L^2([0, 1]; \mathbb{R}^n)$ 
under the equivalence relation $\sim$, where $S\sim T$ if and only if 
$T_\sharp \mathcal{H}^1|_{[0,1]}=S_\sharp \mathcal{H}^1|_{[0,1]}$.
In particular, if $p=2$, 
the map from $L^2([0,1];\mathbb{R}^n)$ to $(\mathcal{P}_2(\R^n),\mk_2^{\mathbb{R}^n})$ sending 
$T$ to $T_\sharp \mathcal{H}^1|_{[0,1]}$ formally becomes a ``Riemannian submersion"
(for instance, see~\cite{Otto01}*{Section~4}).
This Riemannian interpolation is recovered for a complete, separable, geodesic space 
by the use of absolutely continuous curves (\cite{AmbrosioGigliSavare08}*{Chapter~8}, for instance).
This enables one to discuss the notion of differentiability on~$(\mathcal{P}_2(\R^n),\mk_2^{\R^n})$,
see also~\cite{GangboTudorascu19} for various notions of differentiability. It may be possible to apply such an approach to the spaces $(\mathcal{P}^\sigma_{p, q}(E), \scrmk_{p, q})$ in certain settings, which is left for a future work.
\end{remark}

\begin{ack} 
The authors would like to thank Guillaume Carlier, Wilfrid Gangbo, Quentin M{\'e}rigot, and Brendan Pass for fruitful discussions.
JK was supported in part by National Science Foundation grant DMS-2246606.
AT was supported in part by JSPS KAKENHI Grant Numbers 19K03494, 24K21513.
\end{ack}

\medskip

\bibliography{Wpq.bib}
\bibliographystyle{alpha}
\end{document}